\pdfoutput=1
\documentclass[11pt,a4paper]{article}
\usepackage[utf8]{inputenc}
\usepackage[T1]{fontenc}
\usepackage{lmodern}
\usepackage{amsmath,amssymb,amsthm}
\usepackage{mathrsfs}
\usepackage{geometry}
\usepackage{enumitem}
\usepackage{xcolor}
\usepackage{hyperref}

\hypersetup{
    colorlinks=true,
    linkcolor=blue,
    citecolor=red,
    urlcolor=cyan,
    pdftitle={Pathwise solutions of the stochastic Burgers equation driven by a multiplicative Hermite noise},
    pdfauthor={Atef Lechiheb}
}

\newtheorem{theorem}{Theorem}[section]
\newtheorem{lemma}[theorem]{Lemma}
\newtheorem{proposition}[theorem]{Proposition}
\newtheorem{corollary}[theorem]{Corollary}
\theoremstyle{definition}
\newtheorem{definition}[theorem]{Definition}
\newtheorem{remark}[theorem]{Remark}
\newtheorem{assumption}[theorem]{Assumption}

\numberwithin{equation}{section}

\newcommand{\R}{\mathbb{R}}
\newcommand{\E}{\mathbb{E}}
\newcommand{\N}{\mathbb{N}}
\newcommand{\PP}{\mathbb{P}}
\newcommand{\D}{\mathbb{D}}
\newcommand{\BUC}{\mathrm{BUC}}
\newcommand{\eps}{\varepsilon}

\newcommand{\dd}{\mathrm{d}}

\begin{document}

\title{Pathwise solutions of the stochastic Burgers equation driven by a multiplicative Hermite noise}
\author{Atef Lechiheb \\ Toulouse School of Economics, Universit\'e Toulouse Capitole \\ \texttt{atef.lechiheb@tse-fr.eu}}
\date{}
\maketitle

\begin{abstract}
We study the viscous Burgers equation on the real line perturbed by a multiplicative noise $\sigma(u)\dot{Z}_t$, where $Z$ is a Hermite process of order $q\ge1$ and Hurst index $H\in(1/2,1)$, and $\sigma$ is a smooth bounded function. The case $q=1$ is fractional Brownian motion and the case $q=2$ is the Rosenblatt process. Since the sample paths of $Z$ are H\"older continuous of every order smaller than $H$, the stochastic term is understood as a Young integral, and the equation is solved path by path. By means of a transformation of Doss--Sussmann type, the equation is reduced to a parabolic equation with random coefficients and quadratic growth in the gradient. We prove global existence and uniqueness of solutions, the pathwise bound $|u(t,x)|\le \|u_0\|_\infty+\|\sigma\|_\infty|Z_t|$, and the convergence of the solutions driven by smooth approximations of $Z$ (a Wong--Zakai type result without correction term). The chaos structure of the noise appears in the integrability of the solution: all moments are finite and $\E\exp(\lambda\sup_x|u(t,x)|^{2/q})<\infty$ for $\lambda$ small. For the Rosenblatt process, and when $\sigma\ge\sigma_0>0$, the exponential moments $\E\, e^{\lambda u(t,x)}$ are finite for small $\lambda$ and infinite for $\lambda\sigma_0 t^H\ge\theta_H$, where $\theta_H$ is an explicit spectral constant, while the left tail is sub-Gaussian. We also show that the temporal H\"older exponent $H$ is optimal, identify the limits of the power variations of $t\mapsto u(t,x)$, and deduce consistent estimators of $H$ and of the noise intensity from one path observed at one spatial point. Finally, for $q=2$, we relate the pathwise solution to a Skorokhod formulation through the forward/Skorokhod correction formula of \v{C}oupek, Duncan and Pasik-Duncan.
\end{abstract}



% ======================================================================
\section{Introduction}
\label{sec:intro}
% ======================================================================

We consider the stochastic Burgers equation
\begin{equation}
\label{eq:burgers}
\partial_t u(t,x) = \nu\, \partial_{xx}^2 u(t,x) + \tfrac12\, \partial_x\big(u(t,x)^2\big) + \sigma\big(u(t,x)\big)\, \dot{Z}_t,
\qquad t>0,\ x\in\R,
\end{equation}
with initial condition $u(0,\cdot)=u_0$. Here $\nu>0$, $\sigma:\R\to\R$ is a smooth bounded function with bounded derivatives, and $Z=(Z_t)_{t\ge 0}$ is a Hermite process of order $q\ge1$ with Hurst index $H\in(1/2,1)$. The noise depends on time only: it acts on the whole line at once, with an intensity modulated by the solution itself.

The Burgers equation is a classical one-dimensional model of nonlinear fluctuating hydrodynamics \cite{Burgers1974}. Its stochastic versions driven by Gaussian noises have been studied by many authors; see for instance \cite{DaPratoDebusscheTemam1994,BertiniCancriniJonaLasinio1994,GyongyNualart1999} for white noise and \cite{WangZengGuo2010} for an additive cylindrical fractional noise on a bounded interval with Dirichlet boundary conditions. Hermite processes are self-similar processes with stationary increments and long memory. They appear as limits in the non-central limit theorems \cite{Taqqu1979,DobrushinMajor1979}, the process of order $q$ lives in the $q$-th Wiener chaos, and all of them have the covariance of fractional Brownian motion; see \cite{MaejimaTudor2007,Tudor2013}. The process of order $1$ is fractional Brownian motion, and the process of order $2$ is the Rosenblatt process \cite{Taqqu1975,Tudor2008}. Stochastic differential equations driven by Hermite processes, and more generally by processes living in a finite Wiener chaos, have recently been studied with a pathwise integral by Loosveldt, Nachit and Nourdin \cite{LNN2026}, who prove Malliavin differentiability and existence of densities for the solutions.

\paragraph{Choice of the integral.} For a multiplicative noise, the meaning of the product $\sigma(u)\dot Z$ has to be fixed. Two natural choices exist. The first one is a Skorokhod integral, defined through Malliavin calculus with respect to the Wiener process underlying $Z$. For the Rosenblatt process, a complete calculus of this type has been developed by \v{C}oupek, Duncan and Pasik-Duncan \cite{CDP2022}. However, the Skorokhod integral with respect to the Rosenblatt process maps $L^{1/H}(0,T;\D^{k,p})$ into $\D^{k-2,p}$ \cite[Proposition 3]{CDP2022}: it loses two Malliavin derivatives. In a Picard scheme for a nonlinear equation this loss cannot be compensated, and we do not know how to construct Skorokhod solutions of \eqref{eq:burgers} for a nonlinear $\sigma$. For Hermite processes of arbitrary order $q$, a divergence-type integral has been constructed only recently by \v{C}oupek, K\v{r}\'{\i}\v{z} and Svoboda \cite{CKS2025}, who also identify the $1/H$-variation of its indefinite integrals. Its integrands are required to have $q$ Malliavin derivatives, and for $q\ge3$ no relation between this integral and a pathwise integral is known. The second choice uses the regularity of the paths. Since $Z$ has H\"older continuous paths of every order smaller than $H$, and $H>1/2$, the integral $\int \sigma(u(s,x))\,\dd Z_s$ can be defined path by path as a Young integral \cite{Young1936}, provided $s\mapsto u(s,x)$ is H\"older continuous of order larger than $1-H$. This is the point of view adopted here, as in \cite{LNN2026} for stochastic differential equations. It is also the one which is stable under smooth approximations of the noise (Theorem \ref{thm:wong-zakai}).

\paragraph{Method.} Because the noise is scalar and depends on time only, the equation can be transformed. Let $\phi(r,z)$ be the flow of the ordinary differential equation $\partial_r\phi=\sigma(\phi)$, $\phi(0,z)=z$. Writing $u(t,x)=\phi(Z_t,v(t,x))$ and using the change of variables formula for Young integrals, we show that $u$ solves \eqref{eq:burgers} if and only if $v$ solves the parabolic equation
\begin{equation}
\label{eq:intro-v}
\partial_t v = \nu\,\partial_{xx}^2 v + A(Z_t,v)\,(\partial_x v)^2 + \phi(Z_t,v)\,\partial_x v, \qquad v(0)=u_0,
\end{equation}
where $A=\nu\,\partial_{zz}^2\phi/\partial_z\phi$. Equation \eqref{eq:intro-v} contains no stochastic integral. Its coefficients depend on $Z_t$ only, and they are H\"older continuous in time. This is a transformation of Doss--Sussmann type \cite{Doss1977,Sussmann1978}. The price to pay is the quadratic term $A(Z_t,v)(\partial_x v)^2$, which is not present in the original equation. A maximum principle gives $\|v(t)\|_\infty\le\|u_0\|_\infty$. A gradient bound of Bernstein type is obtained on short time windows. On each window, the equation is transformed again, starting from the current time, so that the coefficient $A$ becomes small (it vanishes at $r=0$).

\paragraph{Related work.} Transformations of this kind are classical for stochastic partial differential equations driven by a Brownian motion in time. They are at the basis of the stochastic viscosity solutions of Lions and Souganidis \cite{LionsSouganidis1998a,LionsSouganidis1998b} and of Buckdahn and Ma \cite{BuckdahnMa2001a,BuckdahnMa2001b}, who consider
\[
\dd u=\big\{\mathcal Au+f(t,x,u,\mathfrak s^*\nabla u)\big\}\dd t+g(t,x,u)\circ\dd B_t
\]
with a Stratonovich integral, where $\mathfrak s$ is the diffusion matrix of the second-order operator $\mathcal A$. In \cite{BuckdahnMa2001a} the nonlinearity $f$ is globally Lipschitz in $(u,\nabla u)$, which excludes the Burgers term $u\,\partial_xu$. For a fractional Brownian motion with $H>1/2$, Maslowski and Nualart \cite{MaslowskiNualart2003} proved existence and uniqueness of mild solutions of evolution equations with Lipschitz and bounded coefficients by means of fractional calculus, and again the Burgers nonlinearity is not covered. In the present paper the nonlinearity $u\,\partial_xu$ is handled through the maximum principle and the Bernstein estimate described above, and the noise is a Hermite process of arbitrary order.

\paragraph{Main results.} Under the assumptions of Section \ref{sec:prelim}, we prove the following.
\begin{enumerate}[label=(\roman*)]
\item Global existence and uniqueness of a pathwise solution, together with the representation $u=\phi(Z_t,v)$ (Theorem \ref{thm:main}).
\item Convergence of the classical solutions driven by smooth approximations of $Z$ to the pathwise solution, without correction term (Theorem \ref{thm:wong-zakai}).
\item The pathwise bound $|u(t,x)|\le\|u_0\|_\infty+\|\sigma\|_\infty|Z_t|$, which gives finite moments of all orders and $\E\exp(\lambda\sup_x|u(t,x)|^{2/q})<\infty$ for $\lambda$ small. For the Rosenblatt process ($q=2$) and $\sigma\ge\sigma_0>0$, the exponential moments of $u(t,x)$ exhibit a transition: $\E\,e^{\lambda u(t,x)}<\infty$ when $\lambda\|\sigma\|_\infty t^H<\theta_H$, and $\E\,e^{\lambda u(t,x)}=\infty$ when $\lambda\sigma_0t^H\ge\theta_H$, while $\E\,e^{-\lambda u(t,x)}<\infty$ for every $\lambda\ge0$ (Theorem \ref{thm:moments}). Here $\theta_H=1/(2\lambda_1)$ and $\lambda_1$ is the largest eigenvalue of the kernel of $Z_1$. For a Gaussian noise all exponential moments are finite, so this asymmetry is a genuinely non-Gaussian effect.
\item Optimal H\"older regularity in time: $t\mapsto u(t,x)$ is H\"older continuous of every order smaller than $H$, and, if $\sigma$ does not vanish, it is almost surely not H\"older continuous of any order larger than $H$ at a given time (Theorem \ref{thm:time-regularity}). Moreover, along dyadic partitions of $[0,T]$ with mesh $h_n$, the normalised power variations satisfy $h_n^{1-pH}V^p_{n,T}(u(\cdot,x))\to\E|Z_1|^p\int_0^T|\sigma(u(s,x))|^p\,\dd s$ in probability for every $p\ge1$ (Proposition \ref{prop:power-u}). For $p=1/H$ this is the pathwise counterpart of the result of \cite{CKS2025} for divergence integrals (Corollary \ref{cor:variation}).
\item Statistical estimation from one path observed at one point $x$ on a fixed interval: for $\sigma=\theta\bar\sigma$ with $\bar\sigma$ known and non-vanishing, we construct consistent estimators of $H$ and of $\theta$ based on quadratic variations (Theorem \ref{thm:estimation}). The estimator of $H$ does not depend on $\theta$, $\bar\sigma$ or $q$.
\item For $q=2$, a formula relating the pathwise solution to a Skorokhod formulation with two explicit correction terms, under the assumptions of \cite[Proposition 4]{CDP2022} (Proposition \ref{prop:skorokhod}).
\end{enumerate}

Let us point out what depends on the order $q$. The existence and uniqueness theory of Sections \ref{sec:flow}--\ref{sec:main} uses only the H\"older regularity of the paths of the noise, and the regularity result of Section \ref{sec:regularity} uses in addition the self-similarity of $Z$ and the fact that $Z_1$ has no atom at $0$. The power variations and the estimators of Section \ref{sec:estimation} use the mixing property of the increments of $Z$. All these results hold for every $q\ge1$, and the pathwise part applies without change to any continuous driving path which is H\"older of order larger than $1/2$. Section \ref{sec:moments} uses the chaos structure of $Z$: hypercontractivity for every $q$, and, for the sharp results, the spectral decomposition of the Rosenblatt distribution, which is specific to $q=2$. Section \ref{sec:skorokhod} relies on the calculus of \cite{CDP2022} and is restricted to $q=2$.

\paragraph{Organisation.} Section \ref{sec:prelim} collects the facts we need on Hermite processes, on Young integration and on the heat semigroup. Section \ref{sec:flow} studies the flow $\phi$. Section \ref{sec:transform} introduces the notion of solution and proves the equivalence between \eqref{eq:burgers} and \eqref{eq:intro-v}. Section \ref{sec:pde} is devoted to the analysis of \eqref{eq:intro-v}. The main well-posedness results are stated in Section \ref{sec:main}. Moments and exponential integrability are studied in Section \ref{sec:moments}, regularity in time in Section \ref{sec:regularity}, power variations and statistical estimation in Section \ref{sec:estimation}, and the relation with Skorokhod integration in Section \ref{sec:skorokhod}. Section \ref{sec:conclusion} contains some concluding remarks, and Appendix \ref{app:heat} the H\"older estimates for the heat equation that we use.

% ======================================================================
\section{Preliminaries}
\label{sec:prelim}
% ======================================================================

\subsection{Notation and assumptions}

For a function $f:\R\to\R$ we write $\|f\|_\infty=\sup_x|f(x)|$. $C^k_b(\R)$ is the space of $k$ times continuously differentiable functions with bounded derivatives up to order $k$, and $\|f\|_{C^k_b}=\sum_{j\le k}\|f^{(j)}\|_\infty$. For $\gamma\in(0,1)$, $C^{k+\gamma}_b(\R)$ is the subspace of $C^k_b(\R)$ of functions whose $k$-th derivative is bounded and globally $\gamma$-H\"older continuous. $\BUC(\R)$ is the space of bounded uniformly continuous functions with the supremum norm, and $\BUC^1(\R)=\{f\in\BUC(\R): f'\in\BUC(\R)\}$ with the norm $\|f\|_{C^1}=\|f\|_\infty+\|f'\|_\infty$.

For an interval $I$ and $\beta\in(0,1]$, $C^\beta(I)$ is the space of $\beta$-H\"older continuous functions on $I$, with seminorm
\[
\|f\|_{\beta;I}=\sup_{s,t\in I,\ s\neq t}\frac{|f(t)-f(s)|}{|t-s|^\beta}.
\]
For $\theta\in(0,1)$ and $0\le a<b$, $C^{\theta/2,\theta}_b([a,b]\times\R)$ denotes the space of bounded functions $f(t,x)$ which are $\theta/2$-H\"older in $t$ and $\theta$-H\"older in $x$, uniformly in the other variable. $C^{1+\theta/2,2+\theta}_b([a,b]\times\R)$ is the space of functions $f$ such that $f,\partial_tf,\partial_xf,\partial^2_{xx}f$ are bounded and continuous, and $\partial_t f,\partial^2_{xx}f\in C^{\theta/2,\theta}_b([a,b]\times\R)$. Finally, $C^{1,2}_b([a,b]\times\R)$ is the space of functions $f$ such that $f,\partial_tf,\partial_xf,\partial^2_{xx}f$ are bounded and continuous on $[a,b]\times\R$. We often write $f_t,f_x,f_{xx}$ for the partial derivatives.

Throughout the paper we fix $\nu>0$ and we work under the following assumptions.

\begin{assumption}
\label{ass:main}
\begin{enumerate}[label=(A\arabic*)]
\item $q\ge1$ is an integer, $H\in(1/2,1)$, and $\beta\in(1/2,H)$ is fixed.
\item $\sigma\in C^3_b(\R)$. We set $L=\|\sigma'\|_\infty$.
\item $u_0\in C^{2+\gamma}_b(\R)$ for some $\gamma\in(0,1)$.
\end{enumerate}
\end{assumption}

\subsection{Hermite processes}
\label{subsec:hermite}

Let $W=(W_y)_{y\in\R}$ be a two-sided Wiener process on a complete probability space $(\Omega,\mathcal F,\PP)$, and assume that $\mathcal F$ is generated by $W$. Let $q\ge1$ be an integer and $H\in(1/2,1)$. The Hermite process of order $q$ and Hurst index $H$ is defined, for $t\ge0$, by
\begin{equation}
\label{eq:def-hermite}
Z_t = c(H,q)\int_{\R^q}\left(\int_0^t \prod_{j=1}^q(s-y_j)_+^{-\left(\frac12+\frac{1-H}{q}\right)}\,\dd s\right)\dd W_{y_1}\cdots\dd W_{y_q},
\end{equation}
where the stochastic integral is a multiple Wiener--It\^o integral of order $q$ and $c(H,q)>0$ is chosen so that $\E Z_1^2=1$; see \cite[Sections 1 and 2]{MaejimaTudor2007} and \cite[Definition 3.1 and Proposition 3.1]{Tudor2013}. The exponent satisfies $\frac12+\frac{1-H}{q}<1$ because $H>\frac12$. For $q=1$, $Z$ is a fractional Brownian motion. For $q=2$ the exponent equals $\frac H2-1$, and \eqref{eq:def-hermite} is the representation of the Rosenblatt process used in \cite[Definition 2]{CDP2022}. We write $Z_t=I_q(f_t)$, where
\[
f_t(y_1,\dots,y_q)=c(H,q)\int_0^t \prod_{j=1}^q(s-y_j)_+^{-\left(\frac12+\frac{1-H}{q}\right)}\,\dd s .
\]
The process $Z$ is $H$-self-similar, has stationary increments, and $\E[Z_tZ_s]=\frac12(t^{2H}+s^{2H}-|t-s|^{2H})$ \cite[Propositions 3.1--3.3]{Tudor2013}. In particular, for every fixed $t\ge0$ and $h\ge0$,
\begin{equation}
\label{eq:self-similar}
Z_{t+h}-Z_t \overset{d}{=} h^H\,Z_1 .
\end{equation}

We recall the hypercontractivity inequality for the $q$-th Wiener chaos: if $F$ belongs to the $q$-th chaos, then $\|F\|_{L^p(\Omega)}\le (p-1)^{q/2}\|F\|_{L^2(\Omega)}$ for every $p\ge2$ \cite[Section 1.4]{Nualart2006}, \cite[Chapter 5]{Janson1997}. Since $\E|Z_t-Z_s|^2=|t-s|^{2H}$, we get
\begin{equation}
\label{eq:hyper-increments}
\E|Z_t-Z_s|^p\le (p-1)^{pq/2}\,|t-s|^{pH},\qquad p\ge2 .
\end{equation}
By the Kolmogorov continuity theorem, $Z$ has a version whose paths are H\"older continuous of every order smaller than $H$ on every compact interval, and $\E\,\|Z\|_{\beta';[0,T]}^p<\infty$ for all $p\ge1$, $T>0$ and $\beta'<H$. We always work with this version. We denote by $\Omega_0$ the event of probability one on which $t\mapsto Z_t$ belongs to $C^{\beta'}([0,T])$ for every $T>0$ and every $\beta'<H$.

\begin{lemma}
\label{lem:hermite-moments}
Let $q\ge1$.
\begin{enumerate}[label=(\alph*)]
\item For every $p\ge2$, $\E|Z_1|^p\le(p-1)^{pq/2}$.
\item For every $c\in[0,\frac{q}{2e})$, $\displaystyle\E\exp\big(c|Z_1|^{2/q}\big)\le e^c+\sum_{k\ge1}\Big(\frac{2ec}{q}\Big)^k<\infty$.
\item $\PP(Z_1=0)=0$.
\end{enumerate}
\end{lemma}

\begin{proof}
(a) is hypercontractivity with $\E Z_1^2=1$.

(b) Let $k\ge0$ be an integer and $p_k=2k/q$. If $p_k\ge2$, then by (a),
\[
\E|Z_1|^{2k/q}=\E|Z_1|^{p_k}\le(p_k-1)^{p_kq/2}=(p_k-1)^k\le\Big(\frac{2k}{q}\Big)^k .
\]
If $p_k<2$, then by Jensen's inequality $\E|Z_1|^{p_k}\le(\E Z_1^2)^{p_k/2}=1$. In both cases $\E|Z_1|^{2k/q}\le1+(2k/q)^k$. By monotone convergence,
\begin{align*}
\E\exp\big(c|Z_1|^{2/q}\big)
&=\sum_{k\ge0}\frac{c^k}{k!}\,\E|Z_1|^{2k/q}
\le\sum_{k\ge0}\frac{c^k}{k!}+\sum_{k\ge1}\frac{c^k}{k!}\Big(\frac{2k}{q}\Big)^k\\
&\le e^c+\sum_{k\ge1}\Big(\frac{2c}{q}\Big)^k\frac{k^k}{k!}
\le e^c+\sum_{k\ge1}\Big(\frac{2ec}{q}\Big)^k ,
\end{align*}
where we used $k^k/k!\le e^k$, which follows from $e^k=\sum_{j\ge0}k^j/j!\ge k^k/k!$. The last series converges because $2ec/q<1$.

(c) $Z_1$ is a nonzero element of the $q$-th Wiener chaos, since $\E Z_1^2=1$. By \cite[Theorem 5.1]{Shigekawa1980}, if $F=\sum_{m=0}^NI_m(g_m)$ with $N\ge1$ and $g_N\neq0$, then the law of $F$ is absolutely continuous with respect to the Lebesgue measure. We apply it with $F=Z_1=I_q(f_1)$. In particular $\PP(Z_1=0)=0$. For $q=1$ this is obvious, and for $q=2$ it also follows from Proposition \ref{prop:spectral}(d) below.
\end{proof}

\paragraph{The Rosenblatt case.} When $q=2$, a sharper description of the law of $Z_1$ is available. It is close to the analysis of the Rosenblatt distribution by Veillette and Taqqu \cite{VeilletteTaqqu2013}, and we give a short proof for completeness.

\begin{proposition}
\label{prop:spectral}
Assume $q=2$. There exist a nonincreasing sequence $\lambda_1\ge\lambda_2\ge\dots\ge0$ with $\lambda_1>0$ and $\sum_k\lambda_k^2=\frac12$, and a sequence $(\xi_k)_{k\ge1}$ of independent standard Gaussian random variables, such that
\begin{equation}
\label{eq:spectral}
Z_1=\sum_{k\ge1}\lambda_k\,(\xi_k^2-1),
\end{equation}
where the series converges almost surely and in $L^2(\Omega)$. Let $\theta_H=1/(2\lambda_1)$, so that $\theta_H\ge1/\sqrt2$. Then:
\begin{enumerate}[label=(\alph*)]
\item for $0\le\theta<\theta_H$, $\displaystyle \E\,e^{\theta Z_1}\le\prod_{k\ge1}\frac{e^{-\theta\lambda_k}}{\sqrt{1-2\theta\lambda_k}}<\infty$;
\item for $\theta\ge\theta_H$, $\E\,e^{\theta Z_1}=+\infty$;
\item for every $\theta\ge0$, $\E\,e^{-\theta Z_1}\le e^{\theta^2/2}$;
\item $\PP(Z_1=0)=0$.
\end{enumerate}
\end{proposition}

\begin{proof}
The kernel $f_1$ belongs to $L^2(\R^2)$ and is symmetric, and $\E Z_1^2=2\|f_1\|_{L^2(\R^2)}^2=1$. Let $\mathsf K$ be the Hilbert--Schmidt operator on $L^2(\R)$ with kernel $f_1$. It is compact and self-adjoint, so there exist an orthonormal family $(e_k)$ and real numbers $(\lambda_k)$ with
\[
f_1=\sum_k\lambda_k\, e_k\otimes e_k \quad\text{in }L^2(\R^2),\qquad \sum_k\lambda_k^2=\|f_1\|^2_{L^2(\R^2)}=\tfrac12 .
\]

\emph{Step 1: the eigenvalues are nonnegative.} Let $h\in C_c^\infty(\R)$ and $\varpi=\frac H2-1\in(-\frac34,-\frac12)$. For $s\in[0,1]$ the function $y\mapsto (s-y)_+^{\varpi}$ is locally integrable, hence
\[
\int_\R (s-y)_+^{\varpi}\,|h(y)|\,\dd y\le \|h\|_\infty\int_{\mathrm{supp}\,h\,\cap(-\infty,s)}(s-y)^{\varpi}\,\dd y\le C_h
\]
uniformly in $s\in[0,1]$. Therefore Fubini's theorem applies and
\begin{align*}
\langle \mathsf Kh,h\rangle_{L^2(\R)}
&=\int_{\R^2} f_1(y_1,y_2)\,h(y_1)h(y_2)\,\dd y_1\dd y_2\\
&=c(H,2)\int_0^1\left(\int_{\R}(s-y_1)_+^{\varpi}h(y_1)\,\dd y_1\right)\left(\int_{\R}(s-y_2)_+^{\varpi}h(y_2)\,\dd y_2\right)\dd s\\
&=c(H,2)\int_0^1\left(\int_{\R}(s-y)_+^{\varpi}h(y)\,\dd y\right)^2\dd s\ \ge 0 .
\end{align*}
Since $C_c^\infty(\R)$ is dense in $L^2(\R)$ and $h\mapsto\langle \mathsf Kh,h\rangle$ is continuous, $\langle \mathsf Kh,h\rangle\ge0$ for all $h\in L^2(\R)$. Hence $\lambda_k=\langle \mathsf Ke_k,e_k\rangle\ge0$. Since $f_1\neq0$, at least one $\lambda_k$ is positive. Discarding the zero eigenvalues, which contribute nothing to $f_1$, and using $\lambda_k\to0$, we arrange the remaining ones in nonincreasing order; then $\lambda_1=\max_k\lambda_k>0$.

\emph{Step 2: the representation \eqref{eq:spectral}.} Let $\xi_k=I_1(e_k)$. The $\xi_k$ are independent standard Gaussian variables, and $I_2(e_k\otimes e_k)=\xi_k^2-1$. By linearity and isometry of $I_2$, the partial sums $S_n=\sum_{k\le n}\lambda_k(\xi_k^2-1)$ converge to $I_2(f_1)=Z_1$ in $L^2(\Omega)$. They form a martingale with $\sup_n\E S_n^2\le 1$, so the convergence also holds almost surely.

\emph{Step 3: proof of (a).} For a standard Gaussian $\xi$ and $c<\frac12$, $\E\,e^{c(\xi^2-1)}=e^{-c}(1-2c)^{-1/2}$. Let $0\le\theta<\theta_H$, so that $2\theta\lambda_k\le2\theta\lambda_1<1$ for all $k$. By Fatou's lemma and independence,
\[
\E\,e^{\theta Z_1}\le\liminf_{n\to\infty}\E\,e^{\theta S_n}=\liminf_{n\to\infty}\prod_{k\le n}\frac{e^{-\theta\lambda_k}}{\sqrt{1-2\theta\lambda_k}}=\prod_{k\ge1}\frac{e^{-\theta\lambda_k}}{\sqrt{1-2\theta\lambda_k}} .
\]
The product converges: for $x=2\theta\lambda_k\in[0,2\theta\lambda_1]$, the logarithm of the $k$-th factor equals $-\frac x2-\frac12\log(1-x)$, and $0\le -\frac x2-\frac12\log(1-x)\le C_\theta\, x^2$ with $C_\theta$ depending only on $2\theta\lambda_1<1$. Since $\sum_k\lambda_k^2<\infty$, the product is finite.

\emph{Step 4: proof of (b).} Write $Z_1=\lambda_1(\xi_1^2-1)+Y$, where $Y=\sum_{k\ge2}\lambda_k(\xi_k^2-1)$ is independent of $\xi_1$. For $\theta\ge\theta_H$, $\E\,e^{\theta\lambda_1(\xi_1^2-1)}=+\infty$ because $\theta\lambda_1\ge\frac12$, and $\E\,e^{\theta Y}\in(0,+\infty]$ because $e^{\theta Y}>0$. Hence $\E\,e^{\theta Z_1}=\E\,e^{\theta\lambda_1(\xi_1^2-1)}\,\E\,e^{\theta Y}=+\infty$.

\emph{Step 5: proof of (c).} By Fatou's lemma as in Step 3,
\[
\E\,e^{-\theta Z_1}\le\prod_{k\ge1}\frac{e^{\theta\lambda_k}}{\sqrt{1+2\theta\lambda_k}} .
\]
Using $\log(1+x)\ge x-\frac{x^2}{2}$ for $x\ge0$, we obtain for each factor
\[
\theta\lambda_k-\tfrac12\log(1+2\theta\lambda_k)\le\theta\lambda_k-\tfrac12\Big(2\theta\lambda_k-2\theta^2\lambda_k^2\Big)=\theta^2\lambda_k^2 .
\]
Summing over $k$ and using $\sum_k\lambda_k^2=\frac12$ gives $\E\,e^{-\theta Z_1}\le\exp(\theta^2/2)$.

\emph{Step 6: proof of (d).} With the decomposition of Step 4, and since $Y$ is independent of $\xi_1$ and $\xi_1^2$ has a continuous distribution,
\[
\PP(Z_1=0)=\E\Big[\PP\big(\lambda_1(\xi_1^2-1)=-Y\,\big|\,Y\big)\Big]=0 .
\]

Finally, $\lambda_1^2\le\sum_k\lambda_k^2=\frac12$ gives $\lambda_1\le 1/\sqrt2$, that is $\theta_H\ge1/\sqrt2$.
\end{proof}

\begin{remark}
\label{rem:tails}
For $q=2$, Proposition \ref{prop:spectral} describes the tails of $Z_1$: the left tail is sub-Gaussian, while the right tail is exactly exponential, in the sense that $\E\,e^{\theta Z_1}$ is finite if and only if $\theta<\theta_H$. In particular, for every $c>0$ and $r>1$, $\E\exp(c|Z_1|^r)=+\infty$. Indeed, there exists $C>0$ such that $e^{c x^r}\ge e^{\theta_H x}-C$ for all $x\ge0$, so that $\E\exp(c|Z_1|^r)\ge\E[e^{\theta_HZ_1}\mathbf 1_{\{Z_1\ge0\}}]-C=+\infty$, by (b) and because $\E[e^{\theta_HZ_1}\mathbf 1_{\{Z_1<0\}}]\le1$.
\end{remark}

\subsection{Young integration}

We recall the basic facts on Young integration \cite{Young1936}; see also \cite{FrizVictoir2010}. Let $f\in C^{\gamma'}([a,b])$ and $g\in C^{\beta'}([a,b])$ with $\gamma'+\beta'>1$. For a partition $\pi=\{a=t_0<t_1<\dots<t_n=b\}$ with mesh $|\pi|$, the Riemann--Stieltjes sums $\sum_i f(t_i)(g(t_{i+1})-g(t_i))$ converge as $|\pi|\to0$. The limit is denoted $\int_a^bf\,\dd g$, and it satisfies the Young--Lo\`eve estimate
\begin{equation}
\label{eq:young-loeve}
\left|\int_s^t f\,\dd g-f(s)\big(g(t)-g(s)\big)\right|\le C_{\gamma'+\beta'}\,\|f\|_{\gamma';[s,t]}\|g\|_{\beta';[s,t]}\,|t-s|^{\gamma'+\beta'},
\qquad a\le s<t\le b,
\end{equation}
with $C_\vartheta=(1-2^{1-\vartheta})^{-1}$. If $g\in C^1([a,b])$, then $\int_a^bf\,\dd g=\int_a^bf(s)g'(s)\,\dd s$. The integral is bilinear, and \eqref{eq:young-loeve} implies
\begin{equation}
\label{eq:young-bound}
\left|\int_a^b f\,\dd g\right|\le\Big(|f(a)|+C_{\gamma'+\beta'}\|f\|_{\gamma';[a,b]}(b-a)^{\gamma'}\Big)\|g\|_{\beta';[a,b]}(b-a)^{\beta'} .
\end{equation}

We need the following change of variables formula. It is classical: for Riemann--Stieltjes integrals with respect to H\"older functions of order larger than $1/2$, see \cite[Theorem 4.3.1 and the remark following it]{Zahle1998}. We include a proof because we use it in a slightly unusual form.

\begin{lemma}
\label{lem:chain-rule}
Let $\varrho\in C^{\beta}([0,T])$ with $\beta>1/2$. Let $Y:[0,T]\to\R$ be given by
\[
Y_t=Y_0+\int_0^t a_s\,\dd s+\int_0^t b_s\,\dd\varrho_s ,
\]
where $a\in C([0,T])$ and $b\in C^{\gamma'}([0,T])$ with $\gamma'+\beta>1$. Let $F\in C^2(\R^2)$ and write $F_1=\partial_1F$, $F_2=\partial_2F$. Then, for every $t\in[0,T]$,
\begin{equation}
\label{eq:chain-rule}
F(\varrho_t,Y_t)=F(\varrho_0,Y_0)+\int_0^t F_1(\varrho_s,Y_s)\,\dd\varrho_s+\int_0^tF_2(\varrho_s,Y_s)\,a_s\,\dd s+\int_0^tF_2(\varrho_s,Y_s)\,b_s\,\dd\varrho_s .
\end{equation}
\end{lemma}

\begin{proof}
\emph{Step 1: regularity of $Y$.} For $0\le s<t\le T$, by \eqref{eq:young-loeve} with $\gamma'+\beta>1$,
\begin{align*}
|Y_t-Y_s|
&\le\int_s^t|a_r|\,\dd r+|b_s|\,|\varrho_t-\varrho_s|+\left|\int_s^tb\,\dd\varrho-b_s(\varrho_t-\varrho_s)\right|\\
&\le\|a\|_\infty|t-s|+\|b\|_\infty\|\varrho\|_{\beta}|t-s|^\beta+C_{\gamma'+\beta}\|b\|_{\gamma'}\|\varrho\|_{\beta}|t-s|^{\gamma'+\beta}\\
&\le C_Y\,|t-s|^\beta,
\end{align*}
where all H\"older norms are taken on $[0,T]$ and $C_Y$ depends on $T$. Hence $X=(\varrho,Y)\in C^\beta([0,T];\R^2)$. Let $K\subset\R^2$ be a compact convex set containing $X([0,T])$. Since $F\in C^2$, the functions $F$, $F_1$, $F_2$ and the second derivatives of $F$ are bounded on $K$, and $F_1,F_2$ are Lipschitz on $K$. Consequently $F_1(X),F_2(X)\in C^\beta([0,T])$ and $F_2(X)b\in C^{\min(\beta,\gamma')}([0,T])$.

\emph{Step 2: Taylor expansion.} Fix $t\in(0,T]$ and a partition $\pi=\{0=t_0<\dots<t_n=t\}$. Write $\Delta\varrho_i=\varrho_{t_{i+1}}-\varrho_{t_i}$, $\Delta Y_i=Y_{t_{i+1}}-Y_{t_i}$ and $X_i=X_{t_i}$. By Taylor's formula on the convex set $K$,
\[
F(X_{i+1})-F(X_i)=F_1(X_i)\Delta\varrho_i+F_2(X_i)\Delta Y_i+\rho_i,\qquad |\rho_i|\le\tfrac12\|D^2F\|_{\infty,K}\big(|\Delta\varrho_i|+|\Delta Y_i|\big)^2 ,
\]
where $\|D^2F\|_{\infty,K}$ is the largest of the suprema over $K$ of the absolute values of the second partial derivatives of $F$. By Step 1, $|\rho_i|\le C|t_{i+1}-t_i|^{2\beta}$, so $\sum_i|\rho_i|\le C\,t\,|\pi|^{2\beta-1}\to0$ since $2\beta>1$.

\emph{Step 3: decomposition of $\Delta Y_i$.} We write
\[
\Delta Y_i=\int_{t_i}^{t_{i+1}}a_r\,\dd r+b_{t_i}\Delta\varrho_i+e_i,\qquad e_i=\int_{t_i}^{t_{i+1}}b\,\dd\varrho-b_{t_i}\Delta\varrho_i,
\]
and by \eqref{eq:young-loeve}, $|e_i|\le C_{\gamma'+\beta}\|b\|_{\gamma'}\|\varrho\|_\beta|t_{i+1}-t_i|^{\gamma'+\beta}$. Therefore
\begin{align*}
F(X_t)-F(X_0)
&=\sum_iF_1(X_i)\Delta\varrho_i+\sum_iF_2(X_i)b_{t_i}\Delta\varrho_i\\
&\quad+\sum_iF_2(X_i)\int_{t_i}^{t_{i+1}}a_r\,\dd r+\sum_iF_2(X_i)\,e_i+\sum_i\rho_i .
\end{align*}

\emph{Step 4: passage to the limit.} As $|\pi|\to0$:
\begin{itemize}
\item $\sum_iF_1(X_i)\Delta\varrho_i\to\int_0^tF_1(X_s)\,\dd\varrho_s$, as Riemann--Stieltjes sums of a Young integral with exponents $\beta+\beta>1$;
\item $\sum_iF_2(X_i)b_{t_i}\Delta\varrho_i\to\int_0^tF_2(X_s)b_s\,\dd\varrho_s$, with exponents $\min(\beta,\gamma')+\beta>1$;
\item $\big|\sum_iF_2(X_i)\int_{t_i}^{t_{i+1}}a_r\,\dd r-\int_0^tF_2(X_r)a_r\,\dd r\big|\le\|a\|_\infty\sum_i\int_{t_i}^{t_{i+1}}|F_2(X_i)-F_2(X_r)|\,\dd r\to0$, by uniform continuity of $F_2(X)$;
\item $\big|\sum_iF_2(X_i)e_i\big|\le\|F_2\|_{\infty,K}\,C_{\gamma'+\beta}\|b\|_{\gamma'}\|\varrho\|_\beta\, t\,|\pi|^{\gamma'+\beta-1}\to0$;
\item $\sum_i\rho_i\to0$ by Step 2.
\end{itemize}
Collecting these limits gives \eqref{eq:chain-rule}.
\end{proof}

\subsection{The heat semigroup and a Gronwall lemma}

Let $G_t(x)=(4\pi\nu t)^{-1/2}\exp(-x^2/(4\nu t))$ and $S(t)f=G_t*f$ for $t>0$, $S(0)f=f$. The family $(S(t))_{t\ge0}$ is a strongly continuous contraction semigroup on $\BUC(\R)$, and $\partial_xS(t)f=S(t)f'$ for $f\in\BUC^1(\R)$. Moreover, since $\int_\R|\partial_xG_t(y)|\,\dd y=2G_t(0)=(\pi\nu t)^{-1/2}$,
\begin{equation}
\label{eq:heat-estimates}
\|S(t)f\|_\infty\le\|f\|_\infty,\qquad \|\partial_xS(t)f\|_\infty\le(\pi\nu t)^{-1/2}\|f\|_\infty,\qquad t>0,\ f\in\BUC(\R).
\end{equation}

\begin{lemma}
\label{lem:gronwall}
Let $T'>0$, $a,b\ge0$, and let $f:[0,T']\to[0,\infty)$ be bounded and measurable with
\[
f(t)\le a+b\int_0^t\big(1+(t-s)^{-1/2}\big)f(s)\,\dd s,\qquad t\in[0,T'].
\]
Then $f(t)\le a\big(1+b(T'+2\sqrt{T'})\big)\exp\big(b^2(T'+4\sqrt{T'}+\pi)\,T'\big)$ for all $t\in[0,T']$.
\end{lemma}

\begin{proof}
Let $k(\tau)=1+\tau^{-1/2}$ for $\tau>0$. Inserting the inequality into itself,
\begin{align*}
f(t)&\le a+b\int_0^tk(t-s)\left(a+b\int_0^sk(s-r)f(r)\,\dd r\right)\dd s\\
&=a\Big(1+b\int_0^tk(t-s)\,\dd s\Big)+b^2\int_0^tf(r)\left(\int_r^tk(t-s)k(s-r)\,\dd s\right)\dd r .
\end{align*}
We have $\int_0^tk(t-s)\,\dd s=t+2\sqrt t\le T'+2\sqrt{T'}$. With $\tau=t-r$ and the Beta integral $\int_0^\tau(\tau-s)^{-1/2}s^{-1/2}\,\dd s=\pi$,
\[
\int_r^tk(t-s)k(s-r)\,\dd s=\int_0^\tau\big(1+(\tau-s)^{-1/2}\big)\big(1+s^{-1/2}\big)\dd s=\tau+4\sqrt\tau+\pi\le T'+4\sqrt{T'}+\pi .
\]
Hence $f(t)\le a\big(1+b(T'+2\sqrt{T'})\big)+b^2(T'+4\sqrt{T'}+\pi)\int_0^tf(r)\,\dd r$, and the classical Gronwall lemma gives the claim.
\end{proof}

% ======================================================================
\section{The flow generated by \texorpdfstring{$\sigma$}{sigma}}
\label{sec:flow}
% ======================================================================

Let $\phi:\R^2\to\R$ be defined by
\begin{equation}
\label{eq:flow}
\partial_r\phi(r,z)=\sigma\big(\phi(r,z)\big),\qquad \phi(0,z)=z,\qquad r,z\in\R .
\end{equation}
Since $\sigma$ is bounded and Lipschitz, $\phi$ is well defined on $\R^2$. We write $\phi_z=\partial_z\phi$, $\phi_{zz}=\partial^2_{zz}\phi$, and we define
\begin{equation}
\label{eq:def-A}
A(r,z)=\nu\,\frac{\phi_{zz}(r,z)}{\phi_z(r,z)} .
\end{equation}

\begin{lemma}
\label{lem:flow}
Under (A2), the following properties hold for all $r,r',z\in\R$.
\begin{enumerate}[label=(\alph*)]
\item $\phi(r+r',z)=\phi(r,\phi(r',z))$, and $\phi(-r,\phi(r,z))=z$. The function $\phi$ belongs to $C^2(\R^2)$, and $z\mapsto\phi(r,z)$ belongs to $C^3(\R)$, with derivatives continuous in $(r,z)$.
\item $|\phi(r,z)-z|\le\|\sigma\|_\infty|r|$.
\item $\displaystyle\phi_z(r,z)=\exp\Big(\int_0^r\sigma'\big(\phi(\rho,z)\big)\,\dd\rho\Big)$, hence $e^{-L|r|}\le\phi_z(r,z)\le e^{L|r|}$, and $z\mapsto\phi(r,z)$ is an increasing bijection of $\R$.
\item $\displaystyle A(r,z)=\nu\int_0^r\sigma''\big(\phi(\rho,z)\big)\,\phi_z(\rho,z)\,\dd\rho$, and
\begin{gather*}
|A(r,z)|\le\nu\|\sigma''\|_\infty e^{L|r|}|r|,\qquad
\partial_rA(r,z)=\nu\,\sigma''\big(\phi(r,z)\big)\phi_z(r,z),\\
|\partial_zA(r,z)|\le\nu\,e^{2L|r|}\big(\|\sigma'''\|_\infty+\|\sigma''\|^2_\infty|r|\big)|r| .
\end{gather*}
\item $\phi_z(r,z)\,\sigma(z)=\sigma\big(\phi(r,z)\big)$ and $\phi_z\big(-r,\phi(r,z)\big)\,\phi_z(r,z)=1$.
\item If $\sigma(z)\ge\sigma_0>0$ for all $z$, then $\phi(r,z)\ge z+\sigma_0r$ for $r\ge0$, $\phi(r,z)\le z-\sigma_0|r|$ for $r\le0$, and $\phi(r,z)\le z+\|\sigma\|_\infty r$ for $r\ge0$.
\end{enumerate}
\end{lemma}

\begin{proof}
(a) The group property follows from the uniqueness of solutions of \eqref{eq:flow}, since both sides solve the same equation in $r$ with the same value at $r=0$. Applying it with $(r,r')$ replaced by $(-r,r)$ gives $\phi(-r,\phi(r,z))=\phi(0,z)=z$. Since $\sigma\in C^3_b$, the theorem on differentiable dependence on initial data shows that $z\mapsto\phi(r,z)$ is $C^3$, with derivatives continuous in $(r,z)$. Moreover $\partial_r\phi=\sigma(\phi)$, $\partial^2_{rr}\phi=\sigma'(\phi)\sigma(\phi)$ and $\partial^2_{rz}\phi=\sigma'(\phi)\phi_z$ are continuous, so $\phi\in C^2(\R^2)$.

(b) $|\phi(r,z)-z|=\big|\int_0^r\sigma(\phi(\rho,z))\,\dd\rho\big|\le\|\sigma\|_\infty|r|$.

(c) Differentiating \eqref{eq:flow} in $z$ gives the linear equation $\partial_r\phi_z=\sigma'(\phi)\phi_z$, $\phi_z(0,z)=1$, whose solution is the stated exponential. The bounds follow from $|\sigma'|\le L$. Since $\phi_z>0$, $\phi(r,\cdot)$ is increasing, and it is a bijection by (a).

(d) Taking the logarithm in (c) and differentiating in $z$,
\begin{equation}
\label{eq:A-formula}
\frac{\phi_{zz}(r,z)}{\phi_z(r,z)}=\partial_z\log\phi_z(r,z)=\int_0^r\sigma''\big(\phi(\rho,z)\big)\phi_z(\rho,z)\,\dd\rho ,
\end{equation}
which is the formula for $A$. The bound on $|A|$ follows from $|\sigma''|\le\|\sigma''\|_\infty$ and $\phi_z(\rho,z)\le e^{L|\rho|}\le e^{L|r|}$ for $|\rho|\le|r|$. The formula for $\partial_rA$ follows by differentiating the integral in $r$. For $\partial_zA$ we differentiate \eqref{eq:A-formula} under the integral sign:
\[
\frac{\partial_zA(r,z)}{\nu}=\int_0^r\Big[\sigma'''\big(\phi(\rho,z)\big)\phi_z(\rho,z)^2+\sigma''\big(\phi(\rho,z)\big)\phi_{zz}(\rho,z)\Big]\dd\rho .
\]
By \eqref{eq:A-formula}, $\phi_{zz}(\rho,z)=\phi_z(\rho,z)\,A(\rho,z)/\nu$, hence $|\phi_{zz}(\rho,z)|\le e^{L|\rho|}\cdot\|\sigma''\|_\infty e^{L|\rho|}|\rho|$. Therefore, for $|\rho|\le|r|$, the integrand is bounded by
\[
\|\sigma'''\|_\infty e^{2L|\rho|}+\|\sigma''\|_\infty\cdot\|\sigma''\|_\infty e^{2L|\rho|}|\rho|\le e^{2L|r|}\big(\|\sigma'''\|_\infty+\|\sigma''\|^2_\infty|r|\big),
\]
and integrating over $\rho$ between $0$ and $r$ gives the bound on $|\partial_zA|$.

(e) Fix $z$ and set $D(r)=\phi_z(r,z)\sigma(z)-\sigma(\phi(r,z))$. By (c) and \eqref{eq:flow},
\begin{align*}
D'(r)&=\sigma'\big(\phi(r,z)\big)\phi_z(r,z)\,\sigma(z)-\sigma'\big(\phi(r,z)\big)\,\sigma\big(\phi(r,z)\big)\\
&=\sigma'\big(\phi(r,z)\big)\,D(r),
\end{align*}
and $D(0)=\sigma(z)-\sigma(z)=0$. Hence $D\equiv0$. The second identity is obtained by differentiating $\phi(-r,\phi(r,z))=z$ in $z$.

(f) For $r\ge0$, $\phi(r,z)-z=\int_0^r\sigma(\phi(\rho,z))\,\dd\rho\in[\sigma_0r,\|\sigma\|_\infty r]$. For $r\le0$, $\phi(r,z)-z=-\int_r^0\sigma(\phi(\rho,z))\,\dd\rho\le-\sigma_0|r|$.
\end{proof}

The following algebraic identity is the heart of the transformation.

\begin{lemma}
\label{lem:algebra}
Let $r\in\R$ and let $V\in C^2(\R)$. Set $U(x)=\phi(r,V(x))$. Then
\begin{equation}
\label{eq:algebra}
\nu U''+UU'=\phi_z(r,V)\Big[\nu V''+A(r,V)\,(V')^2+\phi(r,V)\,V'\Big].
\end{equation}
\end{lemma}

\begin{proof}
By the chain rule, $U'=\phi_z(r,V)V'$ and $U''=\phi_{zz}(r,V)(V')^2+\phi_z(r,V)V''$. Hence
\begin{align*}
\nu U''+UU'
&=\nu\,\phi_{zz}(r,V)(V')^2+\nu\,\phi_z(r,V)V''+\phi(r,V)\,\phi_z(r,V)V'\\
&=\phi_z(r,V)\Big[\nu V''+\nu\frac{\phi_{zz}(r,V)}{\phi_z(r,V)}(V')^2+\phi(r,V)V'\Big]\\
&=\phi_z(r,V)\Big[\nu V''+A(r,V)(V')^2+\phi(r,V)V'\Big],
\end{align*}
where we used $\phi_z>0$ and the definition \eqref{eq:def-A} of $A$.
\end{proof}

% ======================================================================
\section{Notion of solution and the transformation}
\label{sec:transform}
% ======================================================================

In this section $\varrho\in C^\beta([0,T])$ is a fixed deterministic path with $\varrho_0=0$ and $\beta>1/2$. In Section \ref{sec:main} we take $\varrho=Z(\omega)$ for $\omega\in\Omega_0$. Note that $\frac12\partial_x(u^2)=uu_x$.

\begin{definition}
\label{def:solution}
Let $u_0$ satisfy (A3). A function $u:[0,T]\times\R\to\R$ is called a \emph{solution of \eqref{eq:burgers} on $[0,T]$ driven by $\varrho$} if:
\begin{enumerate}[label=(\roman*)]
\item $u$, $u_x$ and $u_{xx}$ exist and are bounded and continuous on $[0,T]\times\R$;
\item there exists $\gamma'>1-\beta$ such that $\sup_{x\in\R}\|u(\cdot,x)\|_{\gamma';[0,T]}<\infty$;
\item for every $(t,x)\in[0,T]\times\R$,
\begin{equation}
\label{eq:integral-form}
u(t,x)=u_0(x)+\int_0^t\Big[\nu\,u_{xx}(s,x)+u(s,x)u_x(s,x)\Big]\dd s+\int_0^t\sigma\big(u(s,x)\big)\,\dd\varrho_s ,
\end{equation}
where the last integral is a Young integral.
\end{enumerate}
\end{definition}

By (i) and (ii), $s\mapsto\sigma(u(s,x))$ belongs to $C^{\gamma'}([0,T])$ with $\gamma'+\beta>1$, so the Young integral in \eqref{eq:integral-form} is well defined. When $\varrho\in C^1([0,T])$, it is the Riemann integral $\int_0^t\sigma(u(s,x))\dot\varrho_s\,\dd s$, and a solution is then a classical solution of $\partial_tu=\nu u_{xx}+uu_x+\sigma(u)\dot\varrho$.

The transformed equation reads
\begin{equation}
\label{eq:v}
\partial_tv=\nu\,v_{xx}+A(\varrho_t,v)\,(v_x)^2+\phi(\varrho_t,v)\,v_x,\qquad v(0,\cdot)=u_0 .
\end{equation}

\begin{proposition}
\label{prop:equivalence}
Assume (A2), (A3), and let $\varrho\in C^\beta([0,T])$ with $\varrho_0=0$, $\beta>1/2$.
\begin{enumerate}[label=(\alph*)]
\item If $v\in C^{1,2}_b([0,T]\times\R)$ is a classical solution of \eqref{eq:v}, then $u(t,x)=\phi(\varrho_t,v(t,x))$ is a solution of \eqref{eq:burgers} on $[0,T]$ driven by $\varrho$, in the sense of Definition \ref{def:solution}, with $\gamma'=\beta$.
\item Conversely, if $u$ is a solution of \eqref{eq:burgers} on $[0,T]$ driven by $\varrho$, then $\tilde v(t,x)=\phi(-\varrho_t,u(t,x))$ belongs to $C^{1,2}_b([0,T]\times\R)$, is a classical solution of \eqref{eq:v}, and $u(t,x)=\phi(\varrho_t,\tilde v(t,x))$.
\end{enumerate}
\end{proposition}

\begin{proof}
Let $S_\varrho=\sup_{t\le T}|\varrho_t|$.

(a) \emph{Regularity.} Let $M_v=\sup|v|$. The functions $\phi,\phi_z,\phi_{zz}$ are continuous, hence bounded on the compact set $[-S_\varrho,S_\varrho]\times[-M_v,M_v]$. Since $u_x=\phi_z(\varrho_t,v)v_x$ and $u_{xx}=\phi_{zz}(\varrho_t,v)v_x^2+\phi_z(\varrho_t,v)v_{xx}$, condition (i) holds. For (ii), for $0\le s<t\le T$,
\begin{align*}
|u(t,x)-u(s,x)|
&\le|\phi(\varrho_t,v(t,x))-\phi(\varrho_s,v(t,x))|+|\phi(\varrho_s,v(t,x))-\phi(\varrho_s,v(s,x))|\\
&\le\|\sigma\|_\infty|\varrho_t-\varrho_s|+e^{LS_\varrho}\|v_t\|_\infty|t-s|\\
&\le\Big(\|\sigma\|_\infty\|\varrho\|_{\beta}+e^{LS_\varrho}\|v_t\|_\infty T^{1-\beta}\Big)|t-s|^\beta,
\end{align*}
where we used $\partial_r\phi=\sigma(\phi)$ and Lemma \ref{lem:flow}(c). Thus (ii) holds with $\gamma'=\beta>1-\beta$.

\emph{Equation.} Fix $x$ and apply Lemma \ref{lem:chain-rule} with $F=\phi\in C^2(\R^2)$, $Y_t=v(t,x)$, $a_s=v_t(s,x)$ (continuous) and $b\equiv0$:
\begin{align*}
u(t,x)-u_0(x)
&=\phi(\varrho_t,v(t,x))-\phi(0,u_0(x))\\
&=\int_0^t\partial_r\phi\big(\varrho_s,v(s,x)\big)\,\dd\varrho_s+\int_0^t\phi_z\big(\varrho_s,v(s,x)\big)\,v_t(s,x)\,\dd s\\
&=\int_0^t\sigma\big(u(s,x)\big)\,\dd\varrho_s+\int_0^t\phi_z\big(\varrho_s,v(s,x)\big)\,v_t(s,x)\,\dd s .
\end{align*}
By \eqref{eq:v} and Lemma \ref{lem:algebra} with $r=\varrho_s$ and $V=v(s,\cdot)$,
\begin{align*}
\phi_z(\varrho_s,v)\,v_t
&=\phi_z(\varrho_s,v)\Big[\nu v_{xx}+A(\varrho_s,v)v_x^2+\phi(\varrho_s,v)v_x\Big]\\
&=\nu u_{xx}(s,x)+u(s,x)u_x(s,x).
\end{align*}
This gives \eqref{eq:integral-form}.

(b) Let $u$ be a solution and set $M_u=\sup|u|$, $a_s=\nu u_{xx}(s,x)+u(s,x)u_x(s,x)$ and $b_s=\sigma(u(s,x))$, for a fixed $x$. Then $a$ is continuous and $b\in C^{\gamma'}([0,T])$, and \eqref{eq:integral-form} says that $Y_t=u(t,x)$ has the form required in Lemma \ref{lem:chain-rule}. Apply the lemma with $F(r,z)=\phi(-r,z)$, which belongs to $C^2(\R^2)$ and satisfies
\[
F_1(r,z)=-(\partial_r\phi)(-r,z)=-\sigma\big(\phi(-r,z)\big),\qquad F_2(r,z)=\phi_z(-r,z).
\]
We obtain
\begin{align*}
\tilde v(t,x)-u_0(x)
&=\int_0^t\Big[F_1\big(\varrho_s,u(s,x)\big)+F_2\big(\varrho_s,u(s,x)\big)\,\sigma\big(u(s,x)\big)\Big]\dd\varrho_s+\int_0^tF_2\big(\varrho_s,u(s,x)\big)\,a_s\,\dd s .
\end{align*}
By Lemma \ref{lem:flow}(e) applied with $-r$ in place of $r$, $F_1(r,z)+F_2(r,z)\sigma(z)=-\sigma(\phi(-r,z))+\phi_z(-r,z)\sigma(z)=0$ for all $(r,z)$. Hence the Young integral vanishes identically and
\begin{equation}
\label{eq:vtilde-t}
\tilde v(t,x)=u_0(x)+\int_0^t\phi_z\big(-\varrho_s,u(s,x)\big)\Big[\nu u_{xx}(s,x)+u(s,x)u_x(s,x)\Big]\dd s .
\end{equation}
The integrand is continuous in $(s,x)$ and bounded, so $\tilde v_t$ exists, is bounded and continuous. Since $z\mapsto\phi(-r,z)$ is $C^3$ with derivatives continuous in $(r,z)$, $\tilde v_x=\phi_z(-\varrho_t,u)u_x$ and $\tilde v_{xx}=\phi_{zz}(-\varrho_t,u)u_x^2+\phi_z(-\varrho_t,u)u_{xx}$ are bounded and continuous on $[0,T]\times\R$. Thus $\tilde v\in C^{1,2}_b([0,T]\times\R)$. By Lemma \ref{lem:flow}(a), $u=\phi(\varrho_t,\tilde v)$. Lemma \ref{lem:algebra} with $r=\varrho_t$ and $V=\tilde v(t,\cdot)$ gives
\[
\nu u_{xx}+uu_x=\phi_z(\varrho_t,\tilde v)\Big[\nu\tilde v_{xx}+A(\varrho_t,\tilde v)\tilde v_x^2+\phi(\varrho_t,\tilde v)\tilde v_x\Big].
\]
Inserting this into \eqref{eq:vtilde-t} and using $\phi_z(-\varrho_t,u)\phi_z(\varrho_t,\tilde v)=\phi_z(-\varrho_t,\phi(\varrho_t,\tilde v))\phi_z(\varrho_t,\tilde v)=1$ (Lemma \ref{lem:flow}(e)), we get
\[
\tilde v_t=\nu\tilde v_{xx}+A(\varrho_t,\tilde v)\tilde v_x^2+\phi(\varrho_t,\tilde v)\tilde v_x,
\]
that is, $\tilde v$ is a classical solution of \eqref{eq:v}.
\end{proof}

% ======================================================================
\section{The transformed equation}
\label{sec:pde}
% ======================================================================

We study the equation
\begin{equation}
\label{eq:general-pde}
\partial_tv=\nu\,v_{xx}+F(t,v,v_x),\qquad v(t_0,\cdot)=v_0,
\end{equation}
on $[t_0,T]\times\R$, where $F:[t_0,T]\times\R\times\R\to\R$ satisfies:
\begin{enumerate}[label=(F\arabic*)]
\item $F$ is continuous, and for every $\rho>0$ there exists $L_\rho$ such that
\[
|F(t,z,p)-F(t,z',p')|\le L_\rho\big(|z-z'|+|p-p'|\big)
\]
for all $t\in[t_0,T]$ and $|z|,|z'|,|p|,|p'|\le\rho$.
\end{enumerate}
We set $M_\rho=\sup\{|F(t,z,p)|:t\in[t_0,T],\ |z|,|p|\le\rho\}$. For \eqref{eq:v}, $t_0=0$ and
\begin{equation}
\label{eq:F-rho}
F_\varrho(t,z,p)=A(\varrho_t,z)\,p^2+\phi(\varrho_t,z)\,p .
\end{equation}
By Lemma \ref{lem:flow}, $A$, $\partial_zA$, $\phi$ and $\phi_z$ are continuous on $\R^2$, so $F_\varrho$ satisfies (F1).

A \emph{mild solution} of \eqref{eq:general-pde} on $[t_0,T']$ is a function $v\in C([t_0,T'];\BUC^1(\R))$ such that
\begin{equation}
\label{eq:mild}
v(t)=S(t-t_0)v_0+\int_{t_0}^tS(t-s)\,F\big(s,v(s),v_x(s)\big)\,\dd s,\qquad t\in[t_0,T'].
\end{equation}

\subsection{Local existence}

\begin{proposition}
\label{prop:local}
Assume (F1) and $v_0\in\BUC^1(\R)$. There exist $T_{\max}\in(t_0,T]$ and a unique function $v$ on $[t_0,T_{\max})$ such that, for every $T'<T_{\max}$, the restriction of $v$ to $[t_0,T']$ is a mild solution. Moreover, either $T_{\max}=T$ and $v$ extends to a mild solution on $[t_0,T]$, or
\[
\limsup_{t\uparrow T_{\max}}\|v(t)\|_{C^1}=+\infty .
\]
The length of the existence interval starting from a datum $v_0$ depends only on $\|v_0\|_{C^1}$, $\nu$ and the functions $\rho\mapsto L_\rho,M_\rho$.
\end{proposition}

\begin{proof}
Let $\rho=2\|v_0\|_{C^1}+1$, $\tau\in(0,T-t_0]$, and let $E_\tau$ be the set of $v\in C([t_0,t_0+\tau];\BUC^1(\R))$ with $\sup_t\|v(t)\|_{C^1}\le\rho$, endowed with the distance $d(v,w)=\sup_t\|v(t)-w(t)\|_{C^1}$. It is a complete metric space. For $v\in E_\tau$, the map $s\mapsto F(s,v(s),v_x(s))$ is continuous from $[t_0,t_0+\tau]$ into $\BUC(\R)$, because $F$ is uniformly continuous on $[t_0,T]\times[-\rho,\rho]^2$. Define $\Gamma v$ by the right-hand side of \eqref{eq:mild}. By \eqref{eq:heat-estimates} and $\partial_xS(t)v_0=S(t)v_0'$,
\begin{align*}
\|\Gamma v(t)\|_{C^1}
&\le\|v_0\|_{C^1}+\int_{t_0}^t\Big(1+\big(\pi\nu(t-s)\big)^{-1/2}\Big)\big\|F(s,v(s),v_x(s))\big\|_\infty\,\dd s\\
&\le\|v_0\|_{C^1}+M_\rho\Big(\tau+2\sqrt{\tau/(\pi\nu)}\Big).
\end{align*}
Similarly, for $v,w\in E_\tau$,
\begin{align*}
\|\Gamma v(t)-\Gamma w(t)\|_{C^1}
&\le\int_{t_0}^t\Big(1+\big(\pi\nu(t-s)\big)^{-1/2}\Big)L_\rho\big(\|v(s)-w(s)\|_\infty+\|v_x(s)-w_x(s)\|_\infty\big)\dd s\\
&\le L_\rho\Big(\tau+2\sqrt{\tau/(\pi\nu)}\Big)\,d(v,w).
\end{align*}
The continuity of $t\mapsto\Gamma v(t)$ in $\BUC^1(\R)$ follows from the strong continuity of $S$ on $\BUC(\R)$ and the integrability of the kernel $(t-s)^{-1/2}$, by a standard argument. Choosing $\tau$ such that $M_\rho(\tau+2\sqrt{\tau/(\pi\nu)})\le\rho/2$ and $L_\rho(\tau+2\sqrt{\tau/(\pi\nu)})\le\frac12$, $\Gamma$ is a contraction of $E_\tau$, and its fixed point is a mild solution on $[t_0,t_0+\tau]$. The choice of $\tau$ depends only on $\|v_0\|_{C^1}$, $\nu$, $L_\rho$ and $M_\rho$.

Uniqueness: if $v,w$ are two mild solutions on $[t_0,T']$, let $\rho$ bound both $C^1$ norms on $[t_0,T']$. The same computation gives
\[
\|v(t)-w(t)\|_{C^1}\le L_\rho\max\big(1,(\pi\nu)^{-1/2}\big)\int_{t_0}^t\big(1+(t-s)^{-1/2}\big)\|v(s)-w(s)\|_{C^1}\,\dd s ,
\]
and Lemma \ref{lem:gronwall} (applied on $[t_0,T']$ after a translation of time) with $a=0$ gives $v=w$.

Gluing local solutions, we obtain a maximal solution on $[t_0,T_{\max})$. If $T_{\max}<T$ and $\|v(t)\|_{C^1}$ stayed bounded by some $\rho_0$ as $t\uparrow T_{\max}$, then restarting from $v(t_1)$ with $t_1$ close to $T_{\max}$ would extend the solution beyond $T_{\max}$, since the existence time depends only on $\rho_0$. This proves the alternative. The case $T_{\max}=T$ is treated in the same way.
\end{proof}

\subsection{Regularity and classical solutions}

The next lemma relies on H\"older estimates for the heat equation on the whole line. Since we only need the operator $\partial_t-\nu\partial^2_{xx}$ with constant coefficients, these estimates are proved in Appendix \ref{app:heat} by explicit computations with the Gaussian kernel.

\begin{lemma}
\label{lem:regularity}
Assume (F1).
\begin{enumerate}[label=(\alph*)]
\item Let $v\in C^{1,2}_b([t_0,T']\times\R)$ be a classical solution of \eqref{eq:general-pde}. Then $v$ is a mild solution on $[t_0,T']$.
\item Assume $v_0\in C^{2+\gamma}_b(\R)$, and that for every $\rho>0$ there exists $C_\rho$ with $|F(t,z,p)-F(s,z,p)|\le C_\rho|t-s|^{\beta}$ for $|z|,|p|\le\rho$. Let $v$ be the maximal mild solution and $T'<T_{\max}$. Then $v\in C^{1+\theta/2,2+\theta}_b([t_0,T']\times\R)$ with $\theta=\min(\gamma,\frac12)$. In particular $v\in C^{1,2}_b([t_0,T']\times\R)$ is a classical solution.
\item Assume in addition that $v_0\in C^{3+\gamma}_b(\R)$ and that $F$ is of class $C^1$ in $(z,p)$, with $\partial_zF$ and $\partial_pF$ locally Lipschitz in $(z,p)$ and locally $\beta$-H\"older in $t$, uniformly on bounded sets. Then $v_x\in C^{1+\theta/2,2+\theta}_b([t_0,T']\times\R)$. In particular $v_{xt}$ and $v_{xxx}$ exist and are bounded and continuous on $[t_0,T']\times\R$.
\end{enumerate}
\end{lemma}

\begin{proof}
(a) Let $g(s,x)=v_t(s,x)-\nu v_{xx}(s,x)=F(s,v(s,x),v_x(s,x))$, which is bounded and continuous. Fix $t\in(t_0,T']$ and $x\in\R$. For $s\in[t_0,t)$,
\[
\frac{\dd}{\dd s}\int_\R G_{t-s}(x-y)\,v(s,y)\,\dd y=\int_\R\Big[-(\partial_\tau G_\tau)\big|_{\tau=t-s}(x-y)\,v(s,y)+G_{t-s}(x-y)\,v_t(s,y)\Big]\dd y .
\]
Differentiation under the integral is justified by the Gaussian decay of $G$ and its derivatives and the boundedness of $v,v_t$. Since $\partial_\tau G_\tau=\nu\,\partial^2_{yy}G_\tau$, two integrations by parts (with vanishing boundary terms, because $v,v_y$ are bounded and $G$ decays) give
\[
\int_\R(\partial_\tau G_\tau)(x-y)\,v(s,y)\,\dd y=\nu\int_\R G_\tau(x-y)\,v_{yy}(s,y)\,\dd y .
\]
Therefore $\frac{\dd}{\dd s}\big(S(t-s)v(s)\big)(x)=\big(S(t-s)g(s)\big)(x)$. Integrating between $t_0$ and $t-\eta$ and letting $\eta\downarrow0$ (using the continuity of $v$ and $\|S(\eta)f-f\|_\infty\to0$ for $f\in\BUC(\R)$, and the fact that $v(t)$ is uniformly continuous since $v_x$ is bounded), we get \eqref{eq:mild}. Since $v_{xx}$ is bounded, $v_x(t)$ is Lipschitz, hence uniformly continuous, and $v\in C([t_0,T'];\BUC^1(\R))$ follows from \eqref{eq:mild}.

(b) We shift time so that $t_0=0$, and write $\tau=T'$. Let $\rho=\sup_{[0,\tau]}\|v(t)\|_{C^1}<\infty$ and $g(t,x)=F(t,v(t,x),v_x(t,x))$, so that $g$ is continuous and $\|g\|_\infty\le M_\rho$. By \eqref{eq:mild}, $v=w+Pg$, where $w(t)=S(t)v_0$ and $Pg$ is the heat potential of Appendix \ref{app:heat}.

\emph{H\"older regularity of $v$ and $v_x$.} By Lemma \ref{lem:heat-initial} with $f=v_0$, and by \eqref{eq:P-hoelder}, \eqref{eq:Px-hoelder} with $\theta'=\theta$, there is a constant $C_1$ such that, for $(t,x),(s,y)\in Q_\tau$,
\begin{align*}
|v(t,x)-v(s,y)|&\le|w(t,x)-w(s,y)|+|Pg(t,x)-Pg(s,y)|\\
&\le\|w_x\|_\infty|x-y|+\|w_t\|_\infty|t-s|+CM_\rho\big(|x-y|+|t-s|^{1/2}\big)\\
&\le C_1\big(|x-y|+|t-s|^{1/2}\big),\\
|v_x(t,x)-v_x(s,y)|&\le\|v_0''\|_\infty\big(|x-y|+a_{0,1}(\nu|t-s|)^{1/2}\big)+CM_\rho\big(|x-y|^\theta+|t-s|^{\theta/2}\big)\\
&\le C_1\big(|x-y|^\theta+|t-s|^{\theta/2}\big)\qquad\text{if }|x-y|\le1 .
\end{align*}
In the last line we used $|x-y|\le|x-y|^\theta$ for $|x-y|\le1$ and $|t-s|^{1/2}\le\tau^{(1-\theta)/2}|t-s|^{\theta/2}$.

\emph{H\"older regularity of $g$.} Let $(t,x),(s,y)\in Q_\tau$ with $|x-y|\le1$. By (F1), the H\"older assumption on $F$ in $t$, and the two bounds above,
\begin{align*}
|g(t,x)-g(s,y)|
&\le\big|F(t,v(t,x),v_x(t,x))-F(s,v(t,x),v_x(t,x))\big|\\
&\quad+\big|F(s,v(t,x),v_x(t,x))-F(s,v(s,y),v_x(s,y))\big|\\
&\le C_\rho|t-s|^\beta+L_\rho\big(|v(t,x)-v(s,y)|+|v_x(t,x)-v_x(s,y)|\big)\\
&\le C_\rho|t-s|^\beta+2L_\rho C_1\big(|x-y|^\theta+|t-s|^{\theta/2}\big)+L_\rho C_1|t-s|^{1/2} .
\end{align*}
Since $\theta/2\le\frac14<\beta$ and $\theta/2<\frac12$, the powers $|t-s|^\beta$ and $|t-s|^{1/2}$ are bounded by a constant times $|t-s|^{\theta/2}$ on $[0,\tau]$. For $|x-y|>1$ we use $|g(t,x)-g(t,y)|\le2M_\rho\le2M_\rho|x-y|^\theta$. Hence $g\in C^{\theta/2,\theta}_b(Q_\tau)$.

\emph{Conclusion.} Since $\theta\le\gamma$ and $v_0\in C^{2+\gamma}_b(\R)$, Lemma \ref{lem:heat-initial} gives $w\in C^{1+\gamma/2,2+\gamma}_b(Q_\tau)\subset C^{1+\theta/2,2+\theta}_b(Q_\tau)$, the inclusion following from the boundedness of the functions involved and from $\tau<\infty$. Lemma \ref{lem:heat-schauder} with $\alpha=\theta$ gives $Pg\in C^{1+\theta/2,2+\theta}_b(Q_\tau)$ and $\partial_tPg=\nu\partial^2_{xx}Pg+g$. Since also $w_t=\nu w_{xx}$, we get $v\in C^{1+\theta/2,2+\theta}_b(Q_\tau)$ and $v_t=\nu v_{xx}+F(t,v,v_x)$, so $v$ is a classical solution.

(c) By (b), $v\in C^{1+\theta/2,2+\theta}_b(Q_\tau)$, so $g$ is differentiable in $x$ and
\[
g_x=\partial_zF(t,v,v_x)\,v_x+\partial_pF(t,v,v_x)\,v_{xx}.
\]
The functions $v$, $v_x$ and $v_{xx}$ are bounded and belong to $C^{\theta/2,\theta}_b(Q_\tau)$, and $\partial_zF$, $\partial_pF$ are locally Lipschitz in $(z,p)$ and locally $\beta$-H\"older in $t$. As in (b), products and compositions of such functions are again in $C^{\theta/2,\theta}_b(Q_\tau)$, so $g_x\in C^{\theta/2,\theta}_b(Q_\tau)$. Since $g$ and $g_x$ are bounded, differentiation under the integrals gives $\partial_xS(t-s)g(s)=S(t-s)g_x(s)$ and $\partial_xPg=P(g_x)$. Hence
\[
v_x(t)=S(t)v_0'+P(g_x)(t),\qquad v_0'\in C^{2+\gamma}_b(\R).
\]
Lemmas \ref{lem:heat-initial} and \ref{lem:heat-schauder}, applied to $v_0'$ and to $g_x$, give $v_x\in C^{1+\theta/2,2+\theta}_b(Q_\tau)$.
\end{proof}

For $F=F_\varrho$ given by \eqref{eq:F-rho}, the H\"older condition in (b) holds: by Lemma \ref{lem:flow}(d), $|A(\varrho_t,z)-A(\varrho_s,z)|\le\sup|\partial_rA|\,|\varrho_t-\varrho_s|$ on compact sets, and similarly for $\phi$. Hence, under (A2) and (A3), the maximal mild solution of \eqref{eq:v} is a classical solution in $C^{1,2}_b([0,T']\times\R)$ for every $T'<T_{\max}$.

\subsection{Maximum principle}

\begin{lemma}
\label{lem:max-principle}
Let $v\in C^{1,2}_b([t_0,T']\times\R)$ satisfy $v_t=\nu v_{xx}+c(t,x)v_x$ on $[t_0,T']\times\R$, where $c$ is bounded. Then $\sup_x|v(t,x)|\le\sup_x|v(t_0,x)|$ for all $t\in[t_0,T']$.
\end{lemma}

\begin{proof}
Let $\chi(x)=(1+x^2)^{1/2}$, so that $|\chi'|\le1$ and $0<\chi''\le1$. Let $K=\nu+\|c\|_\infty+1$, $m_0=\sup_xv(t_0,x)$, and, for $\eps>0$,
\[
V_\eps(t,x)=v(t,x)-m_0-\eps\big(\chi(x)+K(t-t_0)\big).
\]
Since $v$ is bounded and $\chi(x)\to\infty$ as $|x|\to\infty$, $V_\eps$ attains its maximum on $[t_0,T']\times\R$ at some point $(t_1,x_1)$. Assume that $V_\eps(t_1,x_1)>0$. Then $t_1>t_0$, because $V_\eps(t_0,\cdot)\le0$. At $(t_1,x_1)$ we have $\partial_tV_\eps\ge0$ (a one-sided derivative if $t_1=T'$), $\partial_xV_\eps=0$ and $\partial^2_{xx}V_\eps\le0$, that is,
\[
v_t\ge\eps K,\qquad v_x=\eps\chi'(x_1),\qquad v_{xx}\le\eps\chi''(x_1).
\]
Using the equation,
\[
\eps K\le v_t=\nu v_{xx}+c\,v_x\le\nu\eps\chi''(x_1)+\|c\|_\infty\eps|\chi'(x_1)|\le\eps(\nu+\|c\|_\infty)<\eps K,
\]
a contradiction. Hence $V_\eps\le0$, that is, $v(t,x)\le m_0+\eps(\chi(x)+K(t-t_0))$. Letting $\eps\to0$ gives $v\le m_0$. Applying this to $-v$ gives the claim.
\end{proof}

\subsection{A gradient estimate of Bernstein type}

We now consider equations of the form
\begin{equation}
\label{eq:bernstein-pde}
\psi_t=\nu\,\psi_{xx}+a(t,\psi)\,\psi_x^2+h(t,\psi)\,\psi_x\qquad\text{on }[t_0,t_1]\times\R,
\end{equation}
where $a,h:[t_0,t_1]\times[-m,m]\to\R$ are continuous, of class $C^1$ in the second variable, with $\partial_za$ and $\partial_zh$ continuous.

\begin{lemma}
\label{lem:bernstein-identity}
Let $\psi$ be a solution of \eqref{eq:bernstein-pde} with values in $[-m,m]$, such that $\psi\in C^{1,2}_b$ and $\zeta=\psi_x\in C^{1,2}_b([t_0,t_1]\times\R)$. Let $g\in C^2([-m,m])$ be positive, $r=g'/g$, and $Y=g(\psi)\,\zeta^2$. Then, at every point where $\zeta\neq0$,
\begin{equation}
\label{eq:bernstein-identity}
Y_t-\nu Y_{xx}=g\Big[-\nu r'+\tfrac{\nu}{2}r^2-a\,r+2\,\partial_za\Big]\zeta^4+2g\,\partial_zh\;\zeta^3+\Big[(2a-\nu r)\zeta+h\Big]Y_x-\frac{\nu}{2g\,\zeta^2}Y_x^2,
\end{equation}
where $g,r,r'$ are evaluated at $\psi$, and $a,h,\partial_za,\partial_zh$ at $(t,\psi)$.
\end{lemma}

\begin{proof}
Differentiating \eqref{eq:bernstein-pde} in $x$ (which is allowed by the assumption $\zeta\in C^{1,2}_b$),
\begin{equation}
\label{eq:q-equation}
\zeta_t-\nu \zeta_{xx}=\big(2a\zeta+h\big)\zeta_x+\partial_za\,\zeta^3+\partial_zh\,\zeta^2 .
\end{equation}
Next,
\[
Y_t=g'\psi_t\,\zeta^2+2g\,\zeta\,\zeta_t,\qquad
Y_x=g'\zeta^3+2g\,\zeta\,\zeta_x,
\]
and
\begin{align*}
Y_{xx}&=g''\zeta^4+3g'\zeta^2\zeta_x+2g'\zeta^2\zeta_x+2g\,\zeta_x^2+2g\,\zeta\,\zeta_{xx}\\
&=g''\zeta^4+5g'\zeta^2\zeta_x+2g\,\zeta_x^2+2g\,\zeta\,\zeta_{xx} .
\end{align*}
Using $\psi_t=\nu \zeta_x+a\zeta^2+h\zeta$ (since $\psi_{xx}=\zeta_x$) and \eqref{eq:q-equation},
\begin{align*}
Y_t-\nu Y_{xx}
&=g'\zeta^2\big(\nu \zeta_x+a\zeta^2+h\zeta\big)+2g\zeta\Big[(2a\zeta+h)\zeta_x+\partial_za\,\zeta^3+\partial_zh\,\zeta^2\Big]\\
&\quad-\nu g''\zeta^4-5\nu g'\zeta^2\zeta_x-2\nu g\,\zeta_x^2\\
&=\big(g'a+2g\,\partial_za-\nu g''\big)\zeta^4+\big(g'h+2g\,\partial_zh\big)\zeta^3\\
&\quad+\big(4ga-4\nu g'\big)\zeta^2\zeta_x+2gh\,\zeta\,\zeta_x-2\nu g\,\zeta_x^2 .
\end{align*}
We now eliminate $\zeta_x$ by means of $\zeta\,\zeta_x=(Y_x-g'\zeta^3)/(2g)$. Using $g'=gr$ and $g''=g(r'+r^2)$:
\begin{align*}
\big(4ga-4\nu g'\big)\zeta^2\zeta_x
&=(2a-2\nu r)\,\zeta\,(Y_x-gr\zeta^3)=(2a-2\nu r)\,\zeta\,Y_x-(2a-2\nu r)\,g\,r\,\zeta^4,\\
2gh\,\zeta\,\zeta_x&=h\,Y_x-h\,g\,r\,\zeta^3,\\
-2\nu g\,\zeta_x^2&=-\frac{2\nu g}{\zeta^2}\big(\zeta\,\zeta_x\big)^2=-\frac{\nu}{2g\,\zeta^2}\big(Y_x-gr\zeta^3\big)^2\\
&=-\frac{\nu}{2g\zeta^2}Y_x^2+\nu r\,\zeta\,Y_x-\frac{\nu}{2}g\,r^2\zeta^4 .
\end{align*}
Collecting the coefficients of $\zeta^4$:
\begin{align*}
&g\Big[ra+2\partial_za-\nu(r'+r^2)-(2a-2\nu r)r-\tfrac{\nu}{2}r^2\Big]
=g\Big[-\nu r'+\tfrac\nu2r^2-a\,r+2\partial_za\Big].
\end{align*}
The coefficient of $\zeta^3$ is $g\,r\,h+2g\,\partial_zh-h\,g\,r=2g\,\partial_zh$. The terms containing $Y_x$ are $(2a-2\nu r)\zeta Y_x+hY_x+\nu r\,\zeta Y_x-\frac{\nu}{2g\zeta^2}Y_x^2=[(2a-\nu r)\zeta+h]Y_x-\frac{\nu}{2g\zeta^2}Y_x^2$. This proves \eqref{eq:bernstein-identity}.
\end{proof}

\begin{proposition}
\label{prop:bernstein}
Let $m\ge1$ and let $\psi$ be as in Lemma \ref{lem:bernstein-identity}. Let
\[
K_h=\sup_{[t_0,t_1]\times[-m,m]}|\partial_zh| ,
\]
and assume the smallness condition
\begin{equation}
\label{eq:smallness}
\sup_{[t_0,t_1]\times[-m,m]}\Big(\frac{|a|}{2m}+2|\partial_za|\Big)\le\frac{\nu}{16m^2}.
\end{equation}
Then, with $\Lambda_0=32\,e^{1/2}m^2K_h/\nu$,
\[
\sup_{x\in\R}|\psi_x(t,x)|\le e^{1/4}\max\Big(\sup_{x\in\R}|\psi_x(t_0,x)|,\ \Lambda_0\Big),\qquad t\in[t_0,t_1].
\]
\end{proposition}

\begin{proof}
\emph{Step 1: choice of $g$.} Let $\mu=1/(4m^2)$ and $g(z)=\exp\big(\frac\mu2(z+m)^2\big)$ for $z\in[-m,m]$. Then $r(z)=\mu(z+m)\in[0,2m\mu]$, $r'=\mu$, and $1\le g\le e^{2\mu m^2}=e^{1/2}$. The coefficient of $\zeta^4$ in \eqref{eq:bernstein-identity}, divided by $g$, satisfies
\begin{align*}
Q&:=-\nu r'+\tfrac\nu2r^2-a\,r+2\,\partial_za\\
&\le-\nu\mu+\tfrac\nu2(2m\mu)^2+2m\mu|a|+2|\partial_za|\\
&=-\frac{\nu}{4m^2}+\frac{\nu}{8m^2}+\frac{|a|}{2m}+2|\partial_za|\\
&\le-\frac{\nu}{8m^2}+\frac{\nu}{16m^2}=-\frac{\nu}{16m^2},
\end{align*}
where we used \eqref{eq:smallness} in the last line.

\emph{Step 2: penalization.} Let $\zeta=\psi_x$ and $Y=g(\psi)\zeta^2$, and set
\[
M_\zeta=\sup_{[t_0,t_1]\times\R}|\zeta|<\infty,\qquad E_h=\sup_{[t_0,t_1]\times[-m,m]}|h|,\qquad \chi(x)=(1+x^2)^{1/2}.
\]
For $\eps>0$, let
\[
Y_\eps(t,x)=Y(t,x)-\eps\big(\chi(x)+\nu(t-t_0)\big).
\] Since $Y$ is bounded, $Y_\eps$ attains its maximum on $[t_0,t_1]\times\R$ at some point $(t_\eps,x_\eps)$. Fix $\eta>0$. We claim that, for $\eps$ small enough (depending on $\eta$),
\begin{equation}
\label{eq:claim-max}
Y_\eps(t_\eps,x_\eps)\le\max\Big(\sup_xY(t_0,x),\ e^{1/2}(\Lambda_0+\eta)^2\Big).
\end{equation}
If $t_\eps=t_0$, then $Y_\eps(t_\eps,x_\eps)\le\sup_xY(t_0,x)$. If $|\zeta(t_\eps,x_\eps)|\le\Lambda_0+\eta$, then $Y_\eps(t_\eps,x_\eps)\le Y(t_\eps,x_\eps)\le e^{1/2}(\Lambda_0+\eta)^2$. It remains to exclude the case $t_\eps>t_0$ and $|\zeta(t_\eps,x_\eps)|>\Lambda_0+\eta$. In that case, at $(t_\eps,x_\eps)$,
\[
Y_t\ge\eps\nu,\qquad Y_x=\eps\chi'(x_\eps),\qquad Y_{xx}\le\eps\chi''(x_\eps)\le\eps,
\]
so $Y_t-\nu Y_{xx}\ge0$. Since $\zeta\neq0$ there, \eqref{eq:bernstein-identity}, Step 1, $|\chi'|\le1$ and the nonpositivity of the last term in \eqref{eq:bernstein-identity} give
\[
0\le g\,Q\,\zeta^4+2g\,K_h|\zeta|^3+\big((2|a|+\nu r)M_\zeta+E_h\big)\eps
\le-\frac{\nu}{16m^2}\,\zeta^4+2e^{1/2}K_h|\zeta|^3+C_\zeta\eps,
\]
with $C_\zeta=(2\sup|a|+2\nu m\mu)M_\zeta+E_h$ and where we used $1\le g\le e^{1/2}$. Writing $s=|\zeta|$, this reads
\[
\frac{\nu}{16m^2}\,s^3\,(s-\Lambda_0)\le C_\zeta\eps .
\]
Since $s>\Lambda_0+\eta$, the left-hand side is at least $\frac{\nu}{16m^2}(\Lambda_0+\eta)^3\eta>0$, which is impossible as soon as $\eps<\frac{\nu(\Lambda_0+\eta)^3\eta}{16m^2C_\zeta}$. This proves \eqref{eq:claim-max}.

\emph{Step 3: conclusion.} For every $(t,x)$ and $\eps$ small,
\[
Y(t,x)\le Y_\eps(t_\eps,x_\eps)+\eps\big(\chi(x)+\nu(t_1-t_0)\big)\le\max\Big(\sup_xY(t_0,x),\ e^{1/2}(\Lambda_0+\eta)^2\Big)+\eps\big(\chi(x)+\nu(t_1-t_0)\big).
\]
Letting $\eps\to0$ and then $\eta\to0$, $Y(t,x)\le\max(\sup_xY(t_0,x),e^{1/2}\Lambda_0^2)$. Since $\zeta^2\le Y\le e^{1/2}\zeta^2$,
\[
\zeta(t,x)^2\le\max\Big(e^{1/2}\sup_x\zeta(t_0,x)^2,\ e^{1/2}\Lambda_0^2\Big),
\]
which is the claim.
\end{proof}

\subsection{Global existence and stability}

We now come back to \eqref{eq:v}. For a path $\varrho\in C^\beta([0,T])$ with $\varrho_0=0$ we use the quantities
\[
S_\varrho=\sup_{t\le T}|\varrho_t|,\qquad\Theta_\varrho=\|\varrho\|_{\beta;[0,T]} .
\]
Given a number
\begin{equation}
\label{eq:constants}
m\ge\max\big(1,\ \|u_0\|_\infty+\|\sigma\|_\infty S_\varrho\big),\qquad\text{we set}\qquad
C_1=\|\sigma''\|_\infty e^{L},\qquad
C_2=e^{2L}\big(\|\sigma'''\|_\infty+\|\sigma''\|_\infty^2\big),
\end{equation}
\begin{equation}
\label{eq:window}
\eps_0=\min\Big(1,\ \frac{1}{16m^2\big(C_1/(2m)+2C_2\big)}\Big),\qquad
\delta=\min\Big(T,\ \big(\eps_0/\Theta_\varrho\big)^{1/\beta}\Big)\quad(\delta=T\text{ if }\Theta_\varrho=0),
\end{equation}
and $\Lambda_0=32\,e^{1/2}m^2e^{L}/\nu$.

\begin{proposition}[A priori bounds for smooth data]
\label{prop:apriori-smooth}
Assume that $\sigma\in C^\infty_b(\R)$, $u_0\in C^\infty_b(\R)$ (all derivatives bounded) and $\varrho\in C^\infty([0,T])$ with $\varrho_0=0$. Let $v$ be the maximal mild solution of \eqref{eq:v} on $[0,T_{\max})$ and $u=\phi(\varrho_t,v)$. Let $m$ satisfy \eqref{eq:constants}, and let $\delta$, $\Lambda_0$ be defined from $m$ by \eqref{eq:window}. Then, for all $t\in[0,T_{\max})$,
\begin{align}
\label{eq:apriori-sup}
&\|v(t)\|_\infty\le\|u_0\|_\infty,\qquad\|u(t)\|_\infty\le\|u_0\|_\infty+\|\sigma\|_\infty|\varrho_t|,\\
\label{eq:apriori-grad}
&\|u_x(t)\|_\infty\le e^{N(L+1/4)}\max\big(\|u_0'\|_\infty,\Lambda_0\big),\qquad N=\lceil T/\delta\rceil .
\end{align}
Consequently $T_{\max}=T$ and $v$ is a mild and classical solution on $[0,T]$.
\end{proposition}

\begin{proof}
Fix $T'<T_{\max}$. By Lemma \ref{lem:regularity}(b),(c), applied to $F_\varrho$ (which is smooth in all variables since $\varrho$ and $\sigma$ are smooth), $v$ and $v_x$ belong to $C^{1,2}_b([0,T']\times\R)$. By Lemma \ref{lem:flow}, $u=\phi(\varrho_t,v)$ has the same regularity, and, $\varrho$ being smooth, a direct computation (or Proposition \ref{prop:equivalence}(a)) shows that $u$ is a classical solution of $u_t=\nu u_{xx}+uu_x+\sigma(u)\dot\varrho_t$.

\emph{Bound \eqref{eq:apriori-sup}.} Equation \eqref{eq:v} can be written $v_t=\nu v_{xx}+c(t,x)v_x$ with $c=A(\varrho_t,v)v_x+\phi(\varrho_t,v)$, which is bounded on $[0,T']\times\R$. Lemma \ref{lem:max-principle} gives $\|v(t)\|_\infty\le\|u_0\|_\infty$. Then Lemma \ref{lem:flow}(b) gives $|u(t,x)|\le|v(t,x)|+\|\sigma\|_\infty|\varrho_t|$.

\emph{Transformation on a window.} Fix $s_0\in[0,T')$ and set $\Delta_t=\varrho_t-\varrho_{s_0}$ for $t\in[s_0,T']$. Define
\[
w(t,x)=\phi\big(-\Delta_t,u(t,x)\big),\qquad t\in[s_0,T'].
\]
Since $\varrho$ is smooth and $u$ is classical,
\[
w_t=-(\partial_r\phi)(-\Delta_t,u)\,\dot\varrho_t+\phi_z(-\Delta_t,u)\,u_t
=-\sigma\big(\phi(-\Delta_t,u)\big)\dot\varrho_t+\phi_z(-\Delta_t,u)\Big[\nu u_{xx}+uu_x+\sigma(u)\dot\varrho_t\Big].
\]
By Lemma \ref{lem:flow}(e), $\phi_z(-\Delta_t,u)\sigma(u)=\sigma(\phi(-\Delta_t,u))$, so the terms with $\dot\varrho_t$ cancel. As in the proof of Proposition \ref{prop:equivalence}(b), Lemma \ref{lem:algebra} and Lemma \ref{lem:flow}(e) give
\begin{equation}
\label{eq:window-pde}
w_t=\nu w_{xx}+A(\Delta_t,w)\,w_x^2+\phi(\Delta_t,w)\,w_x,\qquad w(s_0)=u(s_0).
\end{equation}
Moreover $w$ and $w_x$ belong to $C^{1,2}_b([s_0,T']\times\R)$, because $u$ and $u_x$ do and $\phi$ is smooth. By Lemma \ref{lem:max-principle} applied to \eqref{eq:window-pde}, and by \eqref{eq:apriori-sup},
\[
\|w(t)\|_\infty\le\|u(s_0)\|_\infty\le\|u_0\|_\infty+\|\sigma\|_\infty S_\varrho\le m .
\]

\emph{Smallness on the window.} Let $t\in[s_0,\min(s_0+\delta,T')]$. Then $|\Delta_t|\le\Theta_\varrho\delta^\beta\le\eps_0\le1$. By Lemma \ref{lem:flow}(d) with $|r|\le1$, the coefficients $a(t,z)=A(\Delta_t,z)$ and $h(t,z)=\phi(\Delta_t,z)$ of \eqref{eq:window-pde} satisfy
\[
|a(t,z)|\le\nu C_1|\Delta_t|,\qquad|\partial_za(t,z)|\le\nu C_2|\Delta_t|,\qquad|\partial_zh(t,z)|=\phi_z(\Delta_t,z)\le e^{L}.
\]
Therefore
\[
\frac{|a|}{2m}+2|\partial_za|\le\nu|\Delta_t|\Big(\frac{C_1}{2m}+2C_2\Big)\le\nu\,\eps_0\Big(\frac{C_1}{2m}+2C_2\Big)\le\frac{\nu}{16m^2},
\]
which is \eqref{eq:smallness}. Proposition \ref{prop:bernstein} (with $K_h\le e^L$, so that its constant is bounded by $\Lambda_0$) gives, for $t\in[s_0,\min(s_0+\delta,T')]$,
\[
\|w_x(t)\|_\infty\le e^{1/4}\max\big(\|w_x(s_0)\|_\infty,\Lambda_0\big)=e^{1/4}\max\big(\|u_x(s_0)\|_\infty,\Lambda_0\big).
\]
Since $u=\phi(\Delta_t,w)$, $u_x=\phi_z(\Delta_t,w)\,w_x$ and $\phi_z(\Delta_t,\cdot)\le e^{L|\Delta_t|}\le e^L$. Hence
\begin{equation}
\label{eq:one-window}
\|u_x(t)\|_\infty\le e^{L+1/4}\max\big(\|u_x(s_0)\|_\infty,\Lambda_0\big),\qquad t\in[s_0,\min(s_0+\delta,T')].
\end{equation}

\emph{Iteration.} Let $M_k=\sup\{\|u_x(t)\|_\infty:t\in[0,\min(k\delta,T')]\}$. Applying \eqref{eq:one-window} with $s_0=k\delta$ gives $M_{k+1}\le e^{L+1/4}\max(M_k,\Lambda_0)$, and $M_0=\|u_0'\|_\infty$. By induction, $\max(M_k,\Lambda_0)\le e^{k(L+1/4)}\max(\|u_0'\|_\infty,\Lambda_0)$, which gives \eqref{eq:apriori-grad}, because $N\delta\ge T\ge T'$.

\emph{Global existence.} The bounds \eqref{eq:apriori-sup}--\eqref{eq:apriori-grad} do not depend on $T'$. Since $v_x=u_x/\phi_z(\varrho_t,v)$ and $\phi_z(\varrho_t,\cdot)\ge e^{-LS_\varrho}$, the norm $\|v(t)\|_{C^1}$ stays bounded on $[0,T_{\max})$. By Proposition \ref{prop:local}, $T_{\max}=T$ and the solution extends to $[0,T]$, and Lemma \ref{lem:regularity}(b) shows that it is classical on $[0,T]$.
\end{proof}

\begin{proposition}[Stability]
\label{prop:stability}
Let $(\sigma^k,u_0^k,\varrho^k)$ and $(\sigma,u_0,\varrho)$ satisfy (A2), (A3) and $\varrho^k,\varrho\in C^\beta([0,T])$, $\varrho^k_0=\varrho_0=0$. Assume that $\sup_k\|\sigma^k\|_{C^3_b}<\infty$, $\sigma^k\to\sigma$ in $C^3$ uniformly on compact sets, $u_0^k\to u_0$ in $\BUC^1(\R)$, and $\varrho^k\to\varrho$ uniformly on $[0,T]$. Let $v^k$, $v$ be the maximal mild solutions of \eqref{eq:v} with data $(\sigma^k,u_0^k,\varrho^k)$ and $(\sigma,u_0,\varrho)$. Let $T'<T_{\max}(\sigma,u_0,\varrho)$ and assume that $\sup_k\sup_{t\le T'}\|v^k(t)\|_{C^1}<\infty$ (with $T'<T_{\max}^k$ for all $k$). Then
\[
\sup_{t\le T'}\|v^k(t)-v(t)\|_{C^1}\longrightarrow0 .
\]
\end{proposition}

\begin{proof}
Let $\phi^k$, $A^k$ be the flow and the coefficient associated with $\sigma^k$, and $F^k=F_{\varrho^k}$ the corresponding nonlinearity \eqref{eq:F-rho}. By continuous dependence of solutions of ODEs on the vector field, $\phi^k\to\phi$, $\phi^k_z\to\phi_z$ and $\phi^k_{zz}\to\phi_{zz}$ uniformly on compact subsets of $\R^2$; indeed these functions solve the system $\partial_r\phi=\sigma(\phi)$, $\partial_r\phi_z=\sigma'(\phi)\phi_z$, $\partial_r\phi_{zz}=\sigma''(\phi)\phi_z^2+\sigma'(\phi)\phi_{zz}$, whose right-hand sides converge locally uniformly, and are locally Lipschitz uniformly in $k$ because $\sup_k\|\sigma^k\|_{C^3_b}<\infty$ on compact sets. Hence $A^k\to A$ uniformly on compact sets, and with $\varrho^k\to\varrho$ uniformly we obtain
\[
\eta_k(\rho):=\sup\big\{|F^k(t,z,p)-F(t,z,p)|:t\le T,\ |z|,|p|\le\rho\big\}\longrightarrow0
\]
for every $\rho>0$. The local Lipschitz constants of $F^k$ in (F1) are bounded uniformly in $k$ on bounded sets.

Let $\rho$ bound $\|v^k(t)\|_{C^1}$ and $\|v(t)\|_{C^1}$ for $t\le T'$ and all $k$, and $D_k(t)=\|v^k(t)-v(t)\|_{C^1}$. Subtracting the mild formulations and using \eqref{eq:heat-estimates},
\begin{align*}
D_k(t)&\le\|u_0^k-u_0\|_{C^1}+\int_0^t\Big(1+\big(\pi\nu(t-s)\big)^{-1/2}\Big)\Big\|F^k(s,v^k,v^k_x)-F(s,v,v_x)\Big\|_\infty\dd s\\
&\le\|u_0^k-u_0\|_{C^1}+\int_0^t\Big(1+\big(\pi\nu(t-s)\big)^{-1/2}\Big)\Big(\eta_k(\rho)+L_\rho\,D_k(s)\Big)\dd s\\
&\le\|u_0^k-u_0\|_{C^1}+\eta_k(\rho)\Big(T'+2\sqrt{T'/(\pi\nu)}\Big)\\
&\quad+L_\rho\big(1+(\pi\nu)^{-1/2}\big)\int_0^t\big(1+(t-s)^{-1/2}\big)D_k(s)\,\dd s,
\end{align*}
where $L_\rho$ is a common Lipschitz constant. Lemma \ref{lem:gronwall} gives $\sup_{t\le T'}D_k(t)\le C\big(\|u_0^k-u_0\|_{C^1}+\eta_k(\rho)\big)\to0$.
\end{proof}

\begin{theorem}[Global solution of the transformed equation]
\label{thm:v-global}
Assume (A2), (A3), and let $\varrho\in C^\beta([0,T])$ with $\varrho_0=0$. Then \eqref{eq:v} has a unique mild solution $v$ on $[0,T]$. It belongs to $C^{1,2}_b([0,T]\times\R)$, it is a classical solution, and $u=\phi(\varrho_t,v)$ satisfies \eqref{eq:apriori-sup}--\eqref{eq:apriori-grad} on $[0,T]$, where $\delta$, $\Lambda_0$ and $N$ are computed from $m=1+\|u_0\|_\infty+\|\sigma\|_\infty S_\varrho$ by \eqref{eq:constants}--\eqref{eq:window}. These bounds depend on $\varrho$ only through $S_\varrho$ and $\Theta_\varrho$.
\end{theorem}

\begin{proof}
\emph{Approximation.} Extend $\varrho$ to $\R$ by $\varrho_t=0$ for $t<0$ and $\varrho_t=\varrho_T$ for $t>T$; the extension has the same $\beta$-H\"older seminorm. Let $(\kappa_k)$ be a standard mollifier on $\R$ and set $\bar\varrho^k=\kappa_k*\varrho$ and $\varrho^k_t=\bar\varrho^k_t-\bar\varrho^k_0$. Then $\varrho^k\in C^\infty([0,T])$, $\varrho^k_0=0$, $\varrho^k\to\varrho$ uniformly on $[0,T]$, $\Theta_{\varrho^k}\le\Theta_\varrho$ and $S_{\varrho^k}\le2S_\varrho$. Let also $\sigma^k=\kappa_k*\sigma\in C^\infty_b(\R)$ and $u_0^k=\kappa_k*u_0\in C^\infty_b(\R)$. Then $\|(\sigma^k)^{(j)}\|_\infty\le\|\sigma^{(j)}\|_\infty$ for $j\le3$, and $\sigma^k\to\sigma$ in $C^3$ uniformly on compact sets because $\sigma'''$ is continuous. Similarly $\|u_0^k\|_\infty\le\|u_0\|_\infty$, $\|(u_0^k)'\|_\infty\le\|u_0'\|_\infty$, and $u_0^k\to u_0$ in $\BUC^1(\R)$, because $u_0'$ is uniformly continuous.

\emph{Uniform bounds.} Let $m=1+\|u_0\|_\infty+\|\sigma\|_\infty S_\varrho$, and let $\delta$, $\Lambda_0$, $N$ be defined from $m$ and $(\sigma,\varrho)$ by \eqref{eq:constants}--\eqref{eq:window}. Since $S_{\varrho^k}\to S_\varrho$, there is $k_0$ such that $S_{\varrho^k}\le S_\varrho+1$ and $\|\sigma\|_\infty S_{\varrho^k}\le\|\sigma\|_\infty S_\varrho+1$ for $k\ge k_0$. For such $k$,
\[
\max\big(1,\ \|u_0^k\|_\infty+\|\sigma^k\|_\infty S_{\varrho^k}\big)\le\max\big(1,\ \|u_0\|_\infty+\|\sigma\|_\infty S_\varrho+1\big)=m,
\]
so $m$ satisfies \eqref{eq:constants} for the data $(\sigma^k,u_0^k,\varrho^k)$. Moreover the constants $L$, $C_1$, $C_2$ computed with $\sigma^k$ are bounded by those computed with $\sigma$, and $\Theta_{\varrho^k}\le\Theta_\varrho$. Hence the numbers $\delta^k$, $\Lambda_0^k$, $N^k$ obtained from \eqref{eq:window} with $m$ and the data $(\sigma^k,\varrho^k)$ satisfy $\delta^k\ge\delta$, $\Lambda_0^k\le\Lambda_0$ and $N^k\le N$. By Proposition \ref{prop:apriori-smooth}, for $k\ge k_0$, $v^k$ is global on $[0,T]$, and, with $u^k=\phi^k(\varrho^k_t,v^k)$,
\begin{equation}
\label{eq:uniform-k}
\begin{aligned}
&\|v^k(t)\|_\infty\le\|u_0\|_\infty,\qquad\|u^k(t)\|_\infty\le\|u_0\|_\infty+\|\sigma\|_\infty|\varrho^k_t|,\\
&\|u^k_x(t)\|_\infty\le e^{N(L+1/4)}\max\big(\|u_0'\|_\infty,\Lambda_0\big)
\end{aligned}
\end{equation}
for $t\le T$. Since $v^k_x=u^k_x/\phi^k_z(\varrho^k_t,v^k)$ and $\phi^k_z(\varrho^k_t,\cdot)\ge e^{-L(S_\varrho+1)}$ by Lemma \ref{lem:flow}(c), we get $\sup_{k\ge k_0}\sup_{t\le T}\|v^k(t)\|_{C^1}\le B<\infty$.

\emph{Passage to the limit.} Let $v$ be the maximal mild solution of \eqref{eq:v} and $T'<T_{\max}$. Proposition \ref{prop:stability} gives $v^k\to v$ in $C([0,T'];\BUC^1(\R))$. Hence $\|v(t)\|_{C^1}\le B$ for $t\le T'$, and this bound does not depend on $T'$. By Proposition \ref{prop:local}, $T_{\max}=T$ and $v$ is a mild solution on $[0,T]$. By Lemma \ref{lem:regularity}(b) it belongs to $C^{1,2}_b([0,T]\times\R)$ and is a classical solution. Applying Proposition \ref{prop:stability} on $[0,T]$, we get $v^k\to v$ in $C([0,T];\BUC^1(\R))$; since moreover $\phi^k\to\phi$ and $\phi^k_z\to\phi_z$ uniformly on compact sets, $u^k\to u$ and $u^k_x\to u_x$ uniformly on $[0,T]\times\R$. Letting $k\to\infty$ in \eqref{eq:uniform-k}, we obtain \eqref{eq:apriori-sup} and \eqref{eq:apriori-grad} for $u$, with the constants $\delta$, $\Lambda_0$, $N$ computed from $m=1+\|u_0\|_\infty+\|\sigma\|_\infty S_\varrho$.
\end{proof}

% ======================================================================
\section{Well-posedness}
\label{sec:main}
% ======================================================================

We now apply the previous results to the paths of the Hermite process. Recall that $\Omega_0$ is the event of probability one on which $Z\in C^{\beta'}([0,T])$ for all $T>0$ and $\beta'<H$; in particular $Z\in C^\beta([0,T])$ for all $T$. Also $Z_0=0$.

\begin{theorem}
\label{thm:main}
Assume (A1)--(A3). For every $\omega\in\Omega_0$ and every $T>0$, there exists a unique solution $u=u(\omega)$ of \eqref{eq:burgers} on $[0,T]$ driven by $Z(\omega)$, in the sense of Definition \ref{def:solution}. These solutions are consistent as $T$ varies, so they define a function $u:[0,\infty)\times\R\to\R$. Moreover:
\begin{enumerate}[label=(\alph*)]
\item $u(t,x)=\phi(Z_t,v(t,x))$, where $v$ is the unique classical solution of \eqref{eq:v} with $\varrho=Z$;
\item for all $t\ge0$, $\sup_x|v(t,x)|\le\|u_0\|_\infty$ and
\begin{equation}
\label{eq:pathwise-bound}
\sup_{x\in\R}|u(t,x)|\le\|u_0\|_\infty+\|\sigma\|_\infty|Z_t| ;
\end{equation}
\item for every $T>0$, $\sup_{t\le T}\|u_x(t)\|_\infty<\infty$, with the explicit bound \eqref{eq:apriori-grad} of Theorem \ref{thm:v-global}, where $\varrho=Z(\omega)|_{[0,T]}$;
\item for every $(t,x)$, $\omega\mapsto u(t,x;\omega)$ is a random variable.
\end{enumerate}
\end{theorem}

\begin{proof}
Fix $\omega\in\Omega_0$ and $T>0$, and let $\varrho=Z(\omega)|_{[0,T]}$. By Theorem \ref{thm:v-global}, \eqref{eq:v} has a classical solution $v\in C^{1,2}_b([0,T]\times\R)$, and by Proposition \ref{prop:equivalence}(a), $u=\phi(\varrho_t,v)$ is a solution. If $\tilde u$ is another solution, Proposition \ref{prop:equivalence}(b) shows that $\tilde v=\phi(-\varrho_t,\tilde u)\in C^{1,2}_b$ is a classical solution of \eqref{eq:v}. By Lemma \ref{lem:regularity}(a), $\tilde v$ is a mild solution on $[0,T]$, hence $\tilde v=v$ by the uniqueness part of Proposition \ref{prop:local}, and $\tilde u=\phi(\varrho_t,\tilde v)=u$. Consistency in $T$ follows from uniqueness. Assertions (a)--(c) follow from Theorem \ref{thm:v-global}.

For (d), fix $T\ge t$ and $n\in\N$, and let $K_n=\{\varrho\in C([0,T]):\varrho_0=0,\ \|\varrho\|_{\beta;[0,T]}\le n\}$. The set $K_n$ is closed in $C([0,T])$, since the H\"older seminorm is lower semicontinuous for uniform convergence. By Proposition \ref{prop:stability} (with $\sigma^k=\sigma$, $u_0^k=u_0$) and the uniform bounds of Theorem \ref{thm:v-global}, which depend on $\varrho$ only through $S_\varrho$ and $\Theta_\varrho$, the map $\varrho\mapsto v(t,x)$ is continuous on $K_n$ for the uniform topology. Hence $\varrho\mapsto u(t,x)=\phi(\varrho_t,v(t,x))$ is Borel on $\bigcup_nK_n$. Since $\omega\mapsto Z(\omega)$ is measurable with values in $C([0,T])$ and $Z(\omega)\in\bigcup_nK_n$ for $\omega\in\Omega_0$, $u(t,x)$ is a random variable.
\end{proof}

\begin{theorem}[Smooth approximations]
\label{thm:wong-zakai}
Assume (A1)--(A3), let $\omega\in\Omega_0$, $T>0$, and let $\varrho^n\in C^1([0,T])$ with $\varrho^n_0=0$ be such that $\varrho^n\to Z(\omega)$ uniformly on $[0,T]$ and $\sup_n\|\varrho^n\|_{\beta;[0,T]}<\infty$. (For instance, the mollifications used in the proof of Theorem \ref{thm:v-global}.) Let $u^n$ be the classical solution of
\[
\partial_tu^n=\nu u^n_{xx}+u^nu^n_x+\sigma(u^n)\,\dot\varrho^n_t,\qquad u^n(0)=u_0,
\]
given by Theorem \ref{thm:main} with $Z$ replaced by $\varrho^n$. Then
\[
\sup_{t\le T}\sup_{x\in\R}\Big(|u^n(t,x)-u(t,x)|+|u^n_x(t,x)-u_x(t,x)|\Big)\longrightarrow0 .
\]
\end{theorem}

\begin{proof}
Theorem \ref{thm:v-global} applies to each $\varrho^n$ (for $\varrho^n\in C^1$, Definition \ref{def:solution} reduces to classical solutions). The bounds of Theorem \ref{thm:v-global} depend on $\varrho^n$ only through $S_{\varrho^n}$ and $\Theta_{\varrho^n}$, which are bounded uniformly in $n$. By Proposition \ref{prop:stability}, $v^n\to v$ in $C([0,T];\BUC^1(\R))$. Since $u^n=\phi(\varrho^n_t,v^n)$, $u^n_x=\phi_z(\varrho^n_t,v^n)v^n_x$, and $\phi,\phi_z$ are uniformly continuous on compact sets, the claim follows.
\end{proof}

\begin{remark}
\label{rem:no-correction}
Theorem \ref{thm:wong-zakai} says that the pathwise solution is the limit of the solutions of the equations driven by regularized noises, and that no correction term appears in the limit. This is a consequence of the classical form of the change of variables formula (Lemma \ref{lem:chain-rule}) for Young integrals. A Skorokhod formulation, on the contrary, differs from the pathwise one by correction terms; see Section \ref{sec:skorokhod}.
\end{remark}

% ======================================================================
\section{Moments and exponential integrability}
\label{sec:moments}
% ======================================================================

\begin{theorem}
\label{thm:moments}
Assume (A1)--(A3), and let $u$ be the solution given by Theorem \ref{thm:main}. Let $\kappa_q=2^{(2/q-1)_+}$, that is $\kappa_1=2$ and $\kappa_q=1$ for $q\ge2$. Then, for every $t\ge0$:
\begin{enumerate}[label=(\alph*)]
\item For every $p\ge2$,
\[
\E\Big[\sup_{x\in\R}|u(t,x)|^p\Big]\le2^{p-1}\Big(\|u_0\|^p_\infty+\|\sigma\|_\infty^p(p-1)^{pq/2}\,t^{pH}\Big),
\]
and, for every $T>0$, $\E\big[\sup_{s\le T,\,x\in\R}|u(s,x)|^p\big]<\infty$.
\item If $\lambda\ge0$ and $\kappa_q\lambda\|\sigma\|_\infty^{2/q}t^{2H/q}<\frac{q}{2e}$, then
\[
\E\Big[\exp\Big(\lambda\sup_{x\in\R}|u(t,x)|^{2/q}\Big)\Big]\le e^{\kappa_q\lambda\|u_0\|_\infty^{2/q}}\ \E\exp\Big(\kappa_q\lambda\|\sigma\|_\infty^{2/q}t^{2H/q}|Z_1|^{2/q}\Big)<\infty .
\]
\end{enumerate}
Assume now that $q=2$, and let $\theta_H$ be the constant of Proposition \ref{prop:spectral}.
\begin{enumerate}[label=(\alph*),resume]
\item If $\lambda\ge0$ and $\lambda\|\sigma\|_\infty t^H<\theta_H$, then
\[
\E\Big[\exp\Big(\lambda\sup_{x\in\R}|u(t,x)|\Big)\Big]\le e^{\lambda\|u_0\|_\infty}\Big(\E\,e^{\lambda\|\sigma\|_\infty t^HZ_1}+e^{\lambda^2\|\sigma\|_\infty^2t^{2H}/2}\Big)<\infty .
\]
\item If moreover $\sigma(z)\ge\sigma_0>0$ for all $z\in\R$, then, for every $x\in\R$:
\begin{enumerate}[label=(d\arabic*)]
\item $\E\,e^{\lambda u(t,x)}=+\infty$ whenever $\lambda\sigma_0t^H\ge\theta_H$;
\item $\E\,e^{-\lambda u(t,x)}\le e^{\lambda\|u_0\|_\infty}\Big(1+e^{\lambda^2\|\sigma\|_\infty^2t^{2H}/2}\Big)<\infty$ for every $\lambda\ge0$.
\end{enumerate}
\end{enumerate}
If $q=2$ and $\sigma(z)\le-\sigma_0<0$ for all $z$, (d) holds with $u$ replaced by $-u$.
\end{theorem}

\begin{proof}
Let $m_0=\|u_0\|_\infty$ and $\Sigma=\|\sigma\|_\infty$. By \eqref{eq:self-similar} with $t$ in place of $h$ and $Z_0=0$, $Z_t\overset{d}{=}t^HZ_1$.

(a) By \eqref{eq:pathwise-bound} and $(a+b)^p\le2^{p-1}(a^p+b^p)$ for $a,b\ge0$,
\[
\sup_x|u(t,x)|^p\le2^{p-1}\big(m_0^p+\Sigma^p|Z_t|^p\big).
\]
By \eqref{eq:hyper-increments} with $s=0$, $\E|Z_t|^p\le(p-1)^{pq/2}t^{pH}$, which gives the first bound. For the second, $\sup_{s\le T}|Z_s|\le T^\beta\|Z\|_{\beta;[0,T]}$, which has finite moments of all orders by the Kolmogorov continuity theorem.

(b) Let $r=2/q$. For $a,b\ge0$ we have $(a+b)^r\le\kappa_q(a^r+b^r)$: if $r\le1$ this is the subadditivity of $y\mapsto y^r$, and if $r=2$ it is $(a+b)^2\le2a^2+2b^2$. By \eqref{eq:pathwise-bound},
\[
\sup_x|u(t,x)|^{2/q}\le\big(m_0+\Sigma|Z_t|\big)^{2/q}\le\kappa_q\Big(m_0^{2/q}+\Sigma^{2/q}|Z_t|^{2/q}\Big).
\]
Since $|Z_t|^{2/q}\overset{d}{=}t^{2H/q}|Z_1|^{2/q}$,
\[
\E\Big[\exp\Big(\lambda\sup_x|u(t,x)|^{2/q}\Big)\Big]\le e^{\kappa_q\lambda m_0^{2/q}}\,\E\exp\Big(\kappa_q\lambda\Sigma^{2/q}t^{2H/q}|Z_1|^{2/q}\Big),
\]
and the right-hand side is finite by Lemma \ref{lem:hermite-moments}(b), applied with $c=\kappa_q\lambda\Sigma^{2/q}t^{2H/q}<\frac q{2e}$.

(c) By \eqref{eq:pathwise-bound}, $\exp(\lambda\sup_x|u(t,x)|)\le e^{\lambda m_0}e^{\lambda\Sigma|Z_t|}$, and $e^{c|y|}\le e^{cy}+e^{-cy}$. Hence, with $\theta=\lambda\Sigma t^H<\theta_H$,
\[
\E\,e^{\lambda\Sigma|Z_t|}\le\E\,e^{\theta Z_1}+\E\,e^{-\theta Z_1}\le\E\,e^{\theta Z_1}+e^{\theta^2/2},
\]
by Proposition \ref{prop:spectral}(a),(c).

(d1) By Theorem \ref{thm:main}(a),(b), $u(t,x)=\phi(Z_t,v(t,x))$ with $v(t,x)\ge-m_0$. Since $\phi(r,\cdot)$ is increasing (Lemma \ref{lem:flow}(c)), $u(t,x)\ge\phi(Z_t,-m_0)$. By Lemma \ref{lem:flow}(f), on the event $\{Z_t\ge0\}$ we have $\phi(Z_t,-m_0)\ge-m_0+\sigma_0Z_t$. Therefore
\[
e^{\lambda u(t,x)}\ge e^{-\lambda m_0}\,e^{\lambda\sigma_0Z_t}\,\mathbf 1_{\{Z_t\ge0\}} .
\]
With $\theta=\lambda\sigma_0t^H\ge\theta_H$,
\[
\E\big[e^{\lambda\sigma_0Z_t}\mathbf 1_{\{Z_t\ge0\}}\big]=\E\big[e^{\theta Z_1}\mathbf 1_{\{Z_1\ge0\}}\big]=\E\,e^{\theta Z_1}-\E\big[e^{\theta Z_1}\mathbf 1_{\{Z_1<0\}}\big]\ge\E\,e^{\theta Z_1}-1=+\infty,
\]
by Proposition \ref{prop:spectral}(b).

(d2) Since $\phi(r,\cdot)$ is increasing and $v(t,x)\ge-m_0$, $-u(t,x)\le-\phi(Z_t,-m_0)$. By Lemma \ref{lem:flow}(f) and (b), $-\phi(r,-m_0)\le m_0-\sigma_0r\le m_0$ for $r\ge0$, and $-\phi(r,-m_0)\le m_0+\Sigma|r|$ for $r<0$. Hence $-u(t,x)\le m_0+\Sigma\,Z_t^-$, where $y^-=\max(-y,0)$, and
\[
\E\,e^{-\lambda u(t,x)}\le e^{\lambda m_0}\,\E\,e^{\lambda\Sigma Z_t^-}\le e^{\lambda m_0}\Big(1+\E\,e^{-\lambda\Sigma Z_t}\Big)\le e^{\lambda m_0}\Big(1+e^{\lambda^2\Sigma^2t^{2H}/2}\Big),
\]
where we used $e^{cy^-}\le1+e^{-cy}$ and Proposition \ref{prop:spectral}(c).

Assume finally that $q=2$ and $\sigma\le-\sigma_0$, and let $\tilde u(t,x)=-u(t,-x)$. Then $\tilde u_t(t,x)=-u_t(t,-x)$, $\tilde u_{xx}(t,x)=-u_{xx}(t,-x)$ and $\tilde u(t,x)\tilde u_x(t,x)=-u(t,-x)u_x(t,-x)$. Hence $\tilde u$ satisfies Definition \ref{def:solution} for the equation \eqref{eq:burgers} with $\sigma$ replaced by $\tilde\sigma(z)=-\sigma(-z)\ge\sigma_0$ and $u_0$ replaced by $\tilde u_0=-u_0(-\cdot)$:
\begin{align*}
\tilde u(t,x)&=-u_0(-x)-\int_0^t\Big[\nu u_{xx}(s,-x)+u(s,-x)u_x(s,-x)\Big]\dd s-\int_0^t\sigma\big(u(s,-x)\big)\,\dd Z_s\\
&=\tilde u_0(x)+\int_0^t\Big[\nu\tilde u_{xx}(s,x)+\tilde u(s,x)\tilde u_x(s,x)\Big]\dd s+\int_0^t\tilde\sigma\big(\tilde u(s,x)\big)\,\dd Z_s .
\end{align*}
By uniqueness, $\tilde u$ is the solution associated with $(\tilde\sigma,\tilde u_0)$, and (d1), (d2) applied to $\tilde u(t,-x)=-u(t,x)$ give the claim for $-u$.
\end{proof}

\begin{remark}
\label{rem:gaussian-comparison}
Let us compare the different orders. For $q=1$ (fractional Brownian motion), Theorem \ref{thm:moments}(b) gives Gaussian integrability $\E\exp(\lambda\sup_x|u(t,x)|^2)<\infty$ for small $\lambda$, so all exponential moments of $u(t,x)$ are finite. For the Rosenblatt process ($q=2$), the right tail of $u(t,x)$ is exponential with a critical exponent of order $\theta_H/(\sigma_0 t^H)$, while its left tail is sub-Gaussian. This asymmetry reflects the asymmetry of the Rosenblatt distribution, whose kernel has nonnegative eigenvalues (Proposition \ref{prop:spectral}). For $q\ge3$, part (b) gives the stretched exponential integrability $\E\exp(\lambda\sup_x|u(t,x)|^{2/q})<\infty$ for small $\lambda$. A threshold result analogous to (d) would require a lower bound on the tail of $Z_1$, which is not provided by the spectral argument of Proposition \ref{prop:spectral}, since a symmetric kernel of order $q\ge3$ cannot be diagonalised in the same way. We leave this question open.
\end{remark}

\begin{remark}
\label{rem:gradient-moments}
The bound \eqref{eq:apriori-grad} on $u_x$ grows like $\exp\big(c\,T\,\|Z\|_{\beta;[0,T]}^{1/\beta}\big)$, because $N\approx T/\delta$ and $\delta^{-1}\approx\Theta_\varrho^{1/\beta}$. For $q\ge2$ we have $1/\beta>1\ge2/q$, so Lemma \ref{lem:hermite-moments}(b) does not ensure that this bound is integrable. For $q=2$ it is not: since $\|Z\|_{\beta;[0,1]}\ge|Z_1|$, Remark \ref{rem:tails} shows that its expectation is infinite. We therefore do not claim any moment bound for $u_x$ when $q\ge2$. Whether $\E\|u_x(t)\|_\infty^p<\infty$ holds is an open question; this is the place where the non-Gaussian tails of $Z$ create a genuine difficulty.
\end{remark}

% ======================================================================
\section{Regularity in time}
\label{sec:regularity}
% ======================================================================

\begin{theorem}
\label{thm:time-regularity}
Assume (A1)--(A3), and let $u$ be the solution of Theorem \ref{thm:main}.
\begin{enumerate}[label=(\alph*)]
\item For every $\omega\in\Omega_0$, $T>0$ and $\beta'<H$, $\sup_{x\in\R}\|u(\cdot,x)\|_{\beta';[0,T]}<\infty$. For every $t\ge0$, $u(t,\cdot)\in C^2_b(\R)$.
\item Assume that $\sigma(z)\neq0$ for all $z\in\R$. Then, for every $t\ge0$, $x\in\R$ and $\beta''\in(H,1)$,
\[
\limsup_{h\downarrow0}\frac{|u(t+h,x)-u(t,x)|}{h^{\beta''}}=+\infty\qquad\text{almost surely.}
\]
\end{enumerate}
\end{theorem}

\begin{proof}
(a) Let $T>0$, $S=\sup_{s\le T}|Z_s|$ and $\ell=\sup_{[0,T]\times\R}|v_t|<\infty$. As in the proof of Proposition \ref{prop:equivalence}(a), for $0\le s<t\le T$,
\[
|u(t,x)-u(s,x)|\le\|\sigma\|_\infty|Z_t-Z_s|+e^{LS}\ell\,|t-s|\le\Big(\|\sigma\|_\infty\|Z\|_{\beta';[0,T]}+e^{LS}\ell\,T^{1-\beta'}\Big)|t-s|^{\beta'} .
\]
The spatial regularity follows from $v(t,\cdot)\in C^2_b(\R)$ and Lemma \ref{lem:flow}(a).

(b) \emph{Step 1: increments of $Z$.} Let $h_n=2^{-n}$ and $Y_n=|Z_{t+h_n}-Z_t|/h_n^{\beta''}$. By \eqref{eq:self-similar}, $Y_n\overset{d}{=}h_n^{H-\beta''}|Z_1|$. Hence, for every $M>0$,
\[
\PP(Y_n\le M)=\PP\big(|Z_1|\le Mh_n^{\beta''-H}\big)\longrightarrow\PP(Z_1=0)=0
\]
by Lemma \ref{lem:hermite-moments}(c). Thus, for every $j\ge1$ there is $n_j>n_{j-1}$ with $\PP(Y_{n_j}\le j)\le2^{-j}$. By the Borel--Cantelli lemma, almost surely $Y_{n_j}>j$ for all $j$ large enough, that is, $Y_{n_j}\to+\infty$ almost surely along the deterministic subsequence $(n_j)$.

\emph{Step 2: expansion of the increments of $u$.} Fix $\omega\in\Omega_0$ such that $Y_{n_j}(\omega)\to\infty$. Let $T=t+1$, $S=\sup_{s\le T}|Z_s|$, $\ell=\sup_{[0,T]\times\R}|v_t|$, $\Phi_1=e^{LS}$. Write $\Delta=Z_{t+h}-Z_t$, $v_h=v(t+h,x)$ and $v_0=v(t,x)$, for $h\in(0,1]$. Then
\[
u(t+h,x)-u(t,x)=\Big[\phi(Z_{t+h},v_h)-\phi(Z_t,v_h)\Big]+\Big[\phi(Z_t,v_h)-\phi(Z_t,v_0)\Big].
\]
For the first bracket, by \eqref{eq:flow},
\begin{align*}
\phi(Z_{t+h},v_h)-\phi(Z_t,v_h)
&=\int_{Z_t}^{Z_{t+h}}\sigma\big(\phi(\rho,v_h)\big)\,\dd\rho\\
&=\sigma\big(\phi(Z_t,v_h)\big)\,\Delta\\
&\quad+\int_{Z_t}^{Z_{t+h}}\Big[\sigma\big(\phi(\rho,v_h)\big)-\sigma\big(\phi(Z_t,v_h)\big)\Big]\dd\rho,
\end{align*}
and, since $|\partial_r\phi|\le\|\sigma\|_\infty$,
\[
\big|\sigma\big(\phi(\rho,z)\big)-\sigma\big(\phi(Z_t,z)\big)\big|\le\|\sigma'\|_\infty\|\sigma\|_\infty|\rho-Z_t| ,
\]
so the last integral is bounded by $\frac12\|\sigma'\|_\infty\|\sigma\|_\infty\Delta^2$. Next,
\begin{align*}
\big|\sigma\big(\phi(Z_t,v_h)\big)-\sigma\big(u(t,x)\big)\big|
&=\big|\sigma\big(\phi(Z_t,v_h)\big)-\sigma\big(\phi(Z_t,v_0)\big)\big|\\
&\le\|\sigma'\|_\infty\Phi_1|v_h-v_0|\le\|\sigma'\|_\infty\Phi_1\ell\,h,
\end{align*}
and the second bracket is bounded by $\Phi_1\ell\,h$. Altogether,
\begin{equation}
\label{eq:increment-expansion}
\Big|u(t+h,x)-u(t,x)-\sigma\big(u(t,x)\big)\Delta\Big|\le\tfrac12\|\sigma'\|_\infty\|\sigma\|_\infty\Delta^2+\|\sigma'\|_\infty\Phi_1\ell\,h|\Delta|+\Phi_1\ell\,h .
\end{equation}

\emph{Step 3: conclusion.} Let $c=|\sigma(u(t,x))|>0$. Since $\Delta\to0$ as $h\to0$, for $h$ small enough the first two terms on the right of \eqref{eq:increment-expansion} are bounded by $\frac c2|\Delta|$. Hence, for $h$ small,
\[
\frac{|u(t+h,x)-u(t,x)|}{h^{\beta''}}\ge\frac c2\,\frac{|\Delta|}{h^{\beta''}}-\Phi_1\ell\,h^{1-\beta''} .
\]
Along $h=h_{n_j}$, the first term tends to $+\infty$ and the second to $0$, since $\beta''<1$.
\end{proof}

\begin{remark}
Theorem \ref{thm:time-regularity} shows that the temporal H\"older exponent of $u$ is exactly the exponent $H$ of the noise. In space, $u(t,\cdot)$ is as smooth as the data allow, because the noise does not depend on $x$.
\end{remark}

% ======================================================================
\section{Power variations and statistical estimation}
\label{sec:estimation}
% ======================================================================

In this section we show that the solution, observed at a single point $x$ along a dyadic grid of a fixed interval $[0,T]$, determines the Hurst index $H$ and the intensity of the noise. For the Hermite Ornstein--Uhlenbeck process, that is, for a linear equation with additive noise $\sigma\,\dd Z_t$, estimators of $H$ and $\sigma$ based on quadratic and $1/H$-variations were constructed by Assaad and Tudor \cite{AssaadTudor2020}, who also obtained limit theorems for the estimator of $H$. Their analysis uses the fact that the solution is a Wiener integral with respect to $Z$, hence an element of the $q$-th chaos. In our setting the noise is multiplicative and the equation is nonlinear, the solution does not belong to a finite sum of chaoses, and the chaos computations are replaced by the pathwise expansion \eqref{eq:increment-expansion}. For $T>0$, $n\ge1$, $p>0$ and a function $f:[0,T]\to\R$, let
\[
h_n=T2^{-n},\qquad t^n_i=i\,h_n\quad(0\le i\le2^n),\qquad
V^{p}_{n,T}(f)=\sum_{i=0}^{2^n-1}\big|f(t^n_{i+1})-f(t^n_i)\big|^{p},
\]
and $c_p=\E|Z_1|^p$, which is finite by Lemma \ref{lem:hermite-moments}(a) and positive by Lemma \ref{lem:hermite-moments}(c). Note that $c_2=1$.

\subsection{Power variations of the noise}

\begin{lemma}
\label{lem:ergodic}
Let $\xi_k=Z_{k+1}-Z_k$, $k\ge0$. Let $m\ge1$ and let $F:\R^{m}\to\R$ be measurable with $\E|F(\xi_0,\dots,\xi_{m-1})|^2<\infty$. Then
\[
\frac1N\sum_{k=0}^{N-1}F(\xi_k,\dots,\xi_{k+m-1})\xrightarrow[N\to\infty]{L^2(\Omega)}\E F(\xi_0,\dots,\xi_{m-1}).
\]
\end{lemma}

\begin{proof}
All the random variables involved are functionals of $W$, so their joint law does not change if we work on the canonical space of $W$. We may therefore assume that $\Omega=C(\R)$, that $W$ is the coordinate process, and that $\PP$ is the law of a two-sided Wiener process vanishing at $0$. Let $\vartheta:\Omega\to\Omega$ be defined by $(\vartheta\omega)(y)=\omega(y+1)-\omega(1)$. Then $\vartheta$ preserves $\PP$. For a function $g$ on $\R^m$, let $(\tau g)(y_1,\dots,y_m)=g(y_1-1,\dots,y_m-1)$. For a step function $g$, the definition of the multiple integral gives $I_m(g)\circ\vartheta=I_m(\tau g)$, and by isometry this holds for every $g\in L^2(\R^m)$.

\emph{Step 1: mixing.} Let $U$ be the operator $UF=F\circ\vartheta$ on $L^2(\Omega)$. It is an isometry and $U(F G)=(UF)(UG)$. We claim that
\begin{equation}
\label{eq:mixing}
\E\big[(U^kF)\,G\big]\xrightarrow[k\to\infty]{}\E F\,\E G,\qquad F,G\in L^2(\Omega).
\end{equation}
Let first $F=I_m(f)$ and $G=I_{m'}(g)$ with $m,m'\ge1$. By orthogonality of the chaoses and the isometry,
\[
\E\big[(U^kF)\,G\big]=\E\big[I_m(\tau^kf)\,I_{m'}(g)\big]=\mathbf 1_{\{m=m'\}}\,m!\,\big\langle\tau^kf,g\big\rangle_{L^2(\R^m)} .
\]
If $f$ and $g$ have compact supports, $\tau^kf$ and $g$ have disjoint supports for $k$ large, so the scalar product vanishes. Since $\|\tau^kf\|=\|f\|$, an approximation of $f$ and $g$ in $L^2(\R^m)$ by compactly supported functions shows that $\langle\tau^kf,g\rangle\to0$ for all $f,g\in L^2(\R^m)$. Hence \eqref{eq:mixing} holds when $F$ and $G$ belong to chaoses of positive order, with limit $0=\E F\,\E G$. It holds trivially if $F$ or $G$ is constant, hence by linearity for finite sums of chaos elements. Such sums are dense in $L^2(\Omega)$ because $\mathcal F$ is generated by $W$. For general $F,G$, choose finite sums $F_\eps,G_\eps$ with $\|F-F_\eps\|_2\le\eps$ and $\|G-G_\eps\|_2\le\eps$. Since $U^k$ is an isometry,
\begin{align*}
\Big|\E\big[(U^kF)G\big]-\E F\,\E G\Big|
&\le\Big|\E\big[(U^k(F-F_\eps))G\big]\Big|+\Big|\E\big[(U^kF_\eps)(G-G_\eps)\big]\Big|\\
&\quad+\Big|\E\big[(U^kF_\eps)G_\eps\big]-\E F_\eps\,\E G_\eps\Big|+\big|\E F_\eps\,\E G_\eps-\E F\,\E G\big|\\
&\le\eps\|G\|_2+\eps\|F_\eps\|_2+\Big|\E\big[(U^kF_\eps)G_\eps\big]-\E F_\eps\,\E G_\eps\Big|+\eps\big(\|G\|_2+\|F_\eps\|_2\big).
\end{align*}
Letting $k\to\infty$ and then $\eps\to0$ gives \eqref{eq:mixing}.

\emph{Step 2: the increments are a stationary sequence driven by $U$.} By the substitution $s=s'+k$,
\[
f_{k+1}(y)-f_k(y)=c(H,q)\int_0^1\prod_{j=1}^q\big(s'-(y_j-k)\big)_+^{-\left(\frac12+\frac{1-H}{q}\right)}\dd s'=(\tau^kf_1)(y),
\]
so $\xi_k=I_q(\tau^kf_1)=U^k\xi_0$. Since $U$ is multiplicative, $F(\xi_k,\dots,\xi_{k+m-1})=U^kX$, where $X=F(\xi_0,\dots,\xi_{m-1})\in L^2(\Omega)$.

\emph{Step 3: mean ergodic theorem.} By the von Neumann mean ergodic theorem for the isometry $U$ \cite[Chapter 1]{Krengel1985}, $\frac1N\sum_{k<N}U^kX$ converges in $L^2(\Omega)$ to the orthogonal projection $X^*$ of $X$ onto the space of $U$-invariant elements. Thus $UX^*=X^*$, and $\E X^*=\E X$ because the constant $1$ is invariant. By \eqref{eq:mixing} with $F=G=X^*$, $\E (X^*)^2=\E[(U^kX^*)X^*]\to(\E X^*)^2$, so $\mathrm{Var}(X^*)=0$ and $X^*=\E X$.
\end{proof}

\begin{lemma}
\label{lem:power-Z}
For every $p>0$ and $T>0$,
\begin{equation}
\label{eq:power-Z}
h_n^{1-pH}\,V^{p}_{n,T}(Z)\xrightarrow[n\to\infty]{L^2(\Omega)}c_p\,T .
\end{equation}
Moreover, for every $0\le a<b\le T$ with $a=jT2^{-m}$ and $b=a+T2^{-m}$ for some integers $m\ge1$ and $0\le j<2^m$,
\begin{equation}
\label{eq:power-Z-local}
h_n^{1-pH}\sum_{i:\ a\le t^n_i<b}\big|Z_{t^n_{i+1}}-Z_{t^n_i}\big|^p\xrightarrow[n\to\infty]{L^2(\Omega)}c_p\,(b-a).
\end{equation}
\end{lemma}

\begin{proof}
Let $N=2^n$. By self-similarity and stationarity of increments, the vector $\big(h_n^{-H}(Z_{t^n_{i+1}}-Z_{t^n_i})\big)_{0\le i<N}$ has the same law as $(\xi_i)_{0\le i<N}$. Indeed, $(Z_{h_ns})_{s\ge0}$ has the law of $(h_n^HZ_s)_{s\ge0}$, and $Z_{t^n_{i+1}}-Z_{t^n_i}=Z_{h_n(i+1)}-Z_{h_ni}$. Therefore
\[
h_n^{1-pH}V^p_{n,T}(Z)=h_n\sum_{i<N}\big|h_n^{-H}(Z_{t^n_{i+1}}-Z_{t^n_i})\big|^p\overset{d}{=}T\cdot\frac1N\sum_{i<N}|\xi_i|^p ,
\]
where we used $h_nN=T$. The quantity $\E\big|h_n^{1-pH}V^p_{n,T}(Z)-c_pT\big|^2$ depends only on the law of $h_n^{1-pH}V^p_{n,T}(Z)$, so it equals $T^2\,\E\big|\frac1N\sum_{i<N}|\xi_i|^p-c_p\big|^2$. This tends to $0$ by Lemma \ref{lem:ergodic} with $m=1$ and $F(z)=|z|^p$, since $\E|\xi_0|^{2p}=c_{2p}<\infty$ and $\E|\xi_0|^p=c_p$. This proves \eqref{eq:power-Z}. For \eqref{eq:power-Z-local}, let $n\ge m$. The points $t^n_i$ in $[a,b]$ are $a+(b-a)k2^{-(n-m)}$, $0\le k\le2^{n-m}$. Since $(Z_{a+s}-Z_a)_{s\ge0}$ has the law of $(Z_s)_{s\ge0}$, the left-hand side of \eqref{eq:power-Z-local} has the law of $h_n^{1-pH}V^p_{n-m,b-a}(Z)$, and the mesh of the dyadic partition of level $n-m$ of $[0,b-a]$ is $(b-a)2^{-(n-m)}=h_n$. The claim follows from \eqref{eq:power-Z} on $[0,b-a]$.
\end{proof}

\begin{remark}
For $p=1/H$ the normalisation disappears, and \eqref{eq:power-Z} is the convergence
\[
V^{1/H}_{n,T}(Z)\to c_{1/H}T .
\]
It was proved in probability by Assaad and Tudor \cite[Proposition 6]{AssaadTudor2020}, for $T=1$ and the uniform partitions $\{i/N\}$, by means of the ergodicity of the increments of $Z$, and in $L^1(\Omega)$ along dyadic partitions in \cite[Lemma 3.2]{CKS2025}. Lemma \ref{lem:ergodic} gives a short self-contained proof of the mixing property, and the argument above gives every power $p>0$ at once.
\end{remark}

\subsection{Power variations of the solution}

\begin{proposition}
\label{prop:power-u}
Assume (A1)--(A3), and let $u$ be the solution of Theorem \ref{thm:main}. For every $p\ge1$, $T>0$ and $x\in\R$,
\[
h_n^{1-pH}\,V^p_{n,T}\big(u(\cdot,x)\big)\xrightarrow[n\to\infty]{\ \PP\ }c_p\int_0^T\big|\sigma\big(u(s,x)\big)\big|^p\,\dd s .
\]
\end{proposition}

\begin{proof}
Fix $p\ge1$, $T>0$ and $x\in\R$. Write $h=h_n$, $t_i=t^n_i$, $N=2^n$, $\Delta_i=Z_{t_{i+1}}-Z_{t_i}$ and $\Delta u_i=u(t_{i+1},x)-u(t_i,x)$.

\emph{Step 1: the weighted power variation of $Z$.} Let $w(s)=|\sigma(u(s,x))|^p$ for $s\in[0,T]$. By Theorem \ref{thm:time-regularity}(a), $w$ is continuous on $[0,T]$ for $\omega\in\Omega_0$. Let $\varpi_w(\delta)=\sup\{|w(s)-w(r)|:\ s,r\in[0,T],\ |s-r|\le\delta\}$, so that $\varpi_w(\delta)\to0$ as $\delta\to0$ almost surely, and let
\[
W_n=h^{1-pH}\sum_{i<N}w(t_i)\,|\Delta_i|^p ,\qquad I=c_p\int_0^Tw(s)\,\dd s .
\]
Fix $m\ge1$, let $a_j=jT2^{-m}$, $b_j=a_j+T2^{-m}$ for $0\le j<2^m$, and let $n\ge m$. Set
\[
V_{n,j}=h^{1-pH}\sum_{i:\ a_j\le t_i<b_j}|\Delta_i|^p .
\]
Every $t_i$ lies in exactly one interval $[a_j,b_j)$, and then $|w(t_i)-w(a_j)|\le\varpi_w(T2^{-m})$. Since $\sum_jV_{n,j}=h^{1-pH}V^p_{n,T}(Z)$ and $\sum_jT2^{-m}=T$,
\begin{align*}
|W_n-I|
&\le\Big|\sum_{j<2^m}\sum_{i:\ a_j\le t_i<b_j}\big(w(t_i)-w(a_j)\big)h^{1-pH}|\Delta_i|^p\Big|
+\Big|\sum_{j<2^m}w(a_j)\big(V_{n,j}-c_p\,T2^{-m}\big)\Big|\\
&\quad+c_p\Big|\sum_{j<2^m}w(a_j)\,T2^{-m}-\int_0^Tw(s)\,\dd s\Big|\\
&\le\varpi_w(T2^{-m})\,h^{1-pH}V^p_{n,T}(Z)
+\|w\|_\infty\sum_{j<2^m}\big|V_{n,j}-c_p\,T2^{-m}\big|
+c_p\,T\,\varpi_w(T2^{-m})\\
&=:J_1+J_2+J_3 .
\end{align*}
By \eqref{eq:power-Z}, $K_0:=\sup_n\E\big[h_n^{1-pH}V^p_{n,T}(Z)\big]<\infty$, and the Markov inequality gives
\[
\sup_n\PP\big(h_n^{1-pH}V^p_{n,T}(Z)>K\big)\le\frac{K_0}{K},\qquad K>0.
\] Let $\eps,\eta>0$ and choose $K$ with $K_0/K\le\eta$. Then
\[
\PP(J_1>\eps)\le\eta+\PP\big(\varpi_w(T2^{-m})>\eps/K\big).
\]
The random variable $\|w\|_\infty$ is almost surely finite, and each of the $2^m$ terms of the sum in $J_2$ tends to $0$ in $L^2(\Omega)$ by \eqref{eq:power-Z-local}. Hence $J_2\to0$ in probability as $n\to\infty$, for fixed $m$. Consequently, for every $m$,
\[
\limsup_{n\to\infty}\PP\big(|W_n-I|>3\eps\big)\le\eta+\PP\big(\varpi_w(T2^{-m})>\eps/K\big)+\PP\big(c_pT\varpi_w(T2^{-m})>\eps\big).
\]
The left-hand side does not depend on $m$. Letting $m\to\infty$, the last two probabilities tend to $0$, because $\varpi_w(T2^{-m})\to0$ almost surely. Since $\eta$ is arbitrary, $W_n\to I$ in probability.

\emph{Step 2: comparison of $u$ with $\sigma(u)Z$.} Choose $\beta'\in(1/2,H)$ such that
\begin{equation}
\label{eq:choice-beta}
H-\beta'<\frac{1-H}{p}.
\end{equation}
Let $S=\sup_{s\le T}|Z_s|$, $\Phi_1=e^{LS}$, $\ell=\sup_{[0,T]\times\R}|v_t|$ and $K'=\|Z\|_{\beta';[0,T]}$, which are finite on $\Omega_0$. The derivation of \eqref{eq:increment-expansion} in the proof of Theorem \ref{thm:time-regularity} uses only the values of $Z$ and $v$ on $[t,t+h]$, so it holds with these constants for all $0\le t<t+h\le T$, in particular for $t=t_i$. Set $A_i=\sigma(u(t_i,x))\Delta_i$ and $R_i=\Delta u_i-A_i$. Since $|\Delta_i|\le K'h^{\beta'}$, $h\le T$ and $2\beta'>1$, we have $h^{2\beta'}\le T^{2\beta'-1}h$ and $h^{1+\beta'}\le T^{\beta'}h$, so that
\begin{align*}
|R_i|
&\le\tfrac12\|\sigma'\|_\infty\|\sigma\|_\infty\Delta_i^2+\|\sigma'\|_\infty\Phi_1\ell\,h|\Delta_i|+\Phi_1\ell\,h\\
&\le\tfrac12\|\sigma'\|_\infty\|\sigma\|_\infty (K')^2h^{2\beta'}+\|\sigma'\|_\infty\Phi_1\ell K'\,h^{1+\beta'}+\Phi_1\ell\,h\\
&\le C_R\,h,\qquad C_R=\tfrac12\|\sigma'\|_\infty\|\sigma\|_\infty (K')^2T^{2\beta'-1}+\|\sigma'\|_\infty\Phi_1\ell K' T^{\beta'}+\Phi_1\ell ,
\end{align*}
and $|A_i|\le\|\sigma\|_\infty K'h^{\beta'}$. For $a,b\in\R$ and $p\ge1$, the mean value theorem gives $\zeta$ between $|a|$ and $|a+b|$ with
\[
\big||a+b|^p-|a|^p\big|=p\,\zeta^{p-1}\big||a+b|-|a|\big|\le p\,(|a|+|b|)^{p-1}|b| .
\]
We apply it with $a=A_i$ and $b=R_i$. Since $|\Delta u_i|^p=|A_i+R_i|^p$ and $|A_i|^p=w(t_i)|\Delta_i|^p$, and using $h\le T^{1-\beta'}h^{\beta'}$ and $Nh=T$,
\begin{align*}
\Big|h^{1-pH}V^p_{n,T}\big(u(\cdot,x)\big)-W_n\Big|
&\le h^{1-pH}\sum_{i<N}p\,\big(|A_i|+|R_i|\big)^{p-1}|R_i|\\
&\le h^{1-pH}\,N\,p\,\big(\|\sigma\|_\infty K'h^{\beta'}+C_Rh\big)^{p-1}C_Rh\\
&\le h^{1-pH}\,N\,p\,\big(\|\sigma\|_\infty K'+C_RT^{1-\beta'}\big)^{p-1}h^{\beta'(p-1)}\,C_Rh\\
&=p\,C_R\,T\,\big(\|\sigma\|_\infty K'+C_RT^{1-\beta'}\big)^{p-1}\,h^{\,1-pH+\beta'(p-1)} .
\end{align*}
The exponent of $h$ is
\[
1-pH+\beta'(p-1)=(1-\beta')-p(H-\beta')>(1-\beta')-(1-H)=H-\beta'>0
\]
by \eqref{eq:choice-beta}. Hence the right-hand side tends to $0$ almost surely. Together with Step 1, this proves the proposition.
\end{proof}

For $p=1/H$ the normalising factor equals $1$, and Proposition \ref{prop:power-u} takes the following form.

\begin{corollary}
\label{cor:variation}
Assume (A1)--(A3). For every $T>0$ and $x\in\R$,
\[
V^{1/H}_{n,T}\big(u(\cdot,x)\big)\xrightarrow[n\to\infty]{\ \PP\ }\E|Z_1|^{1/H}\int_0^T\big|\sigma\big(u(s,x)\big)\big|^{1/H}\,\dd s .
\]
\end{corollary}

\begin{remark}
\label{rem:variation}
Corollary \ref{cor:variation} has the same form as \cite[Theorem 3.1]{CKS2025}, where it is shown that
\[
V^{1/H}_{n,T}\Big(\int_0^\cdot g_s\,\delta Z_s\Big)\xrightarrow[n\to\infty]{L^1(\Omega)}\E|Z_1|^{1/H}\int_0^T|g_s|^{1/H}\,\dd s
\]
for integrands $g$ in $L^{1/H}(0,T;\D^{q,1/H})$. In Corollary \ref{cor:variation}, no Malliavin differentiability of $\sigma(u(\cdot,x))$ is needed. The convergence is only in probability, because the moments of the random constant $C_R$ are not controlled (Remark \ref{rem:gradient-moments}).
\end{remark}

\subsection{Estimation of the Hurst index and of the noise intensity}

We now consider the parametric model
\begin{equation}
\label{eq:param-model}
\sigma=\theta\,\bar\sigma,\qquad\theta>0,
\end{equation}
where $\bar\sigma\in C^3_b(\R)$ is known and $\bar\sigma(z)\neq0$ for every $z\in\R$. Then $\sigma$ satisfies (A2) for every $\theta>0$. We denote by $u^\theta$ the solution of Theorem \ref{thm:main} with this $\sigma$, and we assume that the values $u^\theta(t^n_i,x)$, $0\le i\le2^n$, are observed for a fixed $x\in\R$ and a fixed $T>0$. The parameters $\theta$ and $H$ are unknown, and the order $q$ need not be known. We define
\[
\widehat H_n=\frac12-\frac12\log_2\frac{V^2_{n+1,T}(u^\theta(\cdot,x))}{V^2_{n,T}(u^\theta(\cdot,x))},
\]
and, when $H$ is known,
\[
\widehat\theta_n=\left(\frac{h_n^{1-2H}\,V^2_{n,T}\big(u^\theta(\cdot,x)\big)}{h_n\sum_{i=0}^{2^n-1}\bar\sigma\big(u^\theta(t^n_i,x)\big)^2}\right)^{1/2}.
\]
Both quantities are well defined on the event where the denominators are positive, whose probability tends to $1$ by the proof below.

\begin{theorem}
\label{thm:estimation}
Assume (A1) and (A3), and let $\sigma$ be given by \eqref{eq:param-model}. For every $x\in\R$ and $T>0$, as $n\to\infty$,
\[
\widehat H_n\xrightarrow{\ \PP\ }H\qquad\text{and}\qquad\widehat\theta_n\xrightarrow{\ \PP\ }\theta .
\]
\end{theorem}

\begin{proof}
Write $u=u^\theta(\cdot,x)$ and
\[
\Lambda=\int_0^T\sigma\big(u(s)\big)^2\,\dd s=\theta^2\int_0^T\bar\sigma\big(u(s)\big)^2\,\dd s .
\]
Since $s\mapsto\bar\sigma(u(s))$ is continuous and does not vanish, $\Lambda>0$ almost surely.

\emph{Hurst index.} By Proposition \ref{prop:power-u} with $p=2$ and $c_2=1$, $X_n:=h_n^{1-2H}V^2_{n,T}(u)\to\Lambda$ in probability. Since $h_{n+1}=h_n/2$,
\[
\frac{V^2_{n+1,T}(u)}{V^2_{n,T}(u)}=\frac{h_{n+1}^{2H-1}X_{n+1}}{h_n^{2H-1}X_n}=2^{1-2H}\,\frac{X_{n+1}}{X_n}.
\]
The pair $(X_n,X_{n+1})$ converges in probability to $(\Lambda,\Lambda)$, and the map $(a,b)\mapsto\log_2(b/a)$ is continuous on $(0,\infty)^2$, which contains $(\Lambda,\Lambda)$ almost surely. By the continuous mapping theorem, $\log_2(X_{n+1}/X_n)\to0$ in probability. Therefore
\[
\widehat H_n=\frac12-\frac12\Big[(1-2H)+\log_2\frac{X_{n+1}}{X_n}\Big]=H-\frac12\log_2\frac{X_{n+1}}{X_n}\xrightarrow{\ \PP\ }H .
\]
In particular $\PP(X_n>0,\,X_{n+1}>0)\to1$, which shows that $\widehat H_n$ is well defined with probability tending to $1$.

\emph{Intensity.} The function $s\mapsto\bar\sigma(u(s))^2$ is continuous on $[0,T]$, so its Riemann sums converge almost surely:
\[
D_n:=h_n\sum_{i=0}^{2^n-1}\bar\sigma\big(u(t^n_i)\big)^2\xrightarrow[n\to\infty]{\text{a.s.}}\int_0^T\bar\sigma\big(u(s)\big)^2\,\dd s=\frac{\Lambda}{\theta^2}>0 .
\]
Hence $X_n/D_n\to\Lambda\cdot\theta^2/\Lambda=\theta^2$ in probability, and $\widehat\theta_n=(X_n/D_n)^{1/2}\to\theta$ by continuity of the square root.
\end{proof}

\begin{remark}
\label{rem:estimation}
(i) The estimator $\widehat H_n$ uses neither $\theta$, nor $\bar\sigma$, nor the order $q$: the unknown factor $\Lambda$ cancels in the ratio of two consecutive quadratic variations. The estimator $-\log S_N/(2\log N)$ of \cite[(40)]{AssaadTudor2020}, where $S_N$ is the normalised quadratic variation, does not depend on $\sigma$ either, but there the factor $\log\sigma^2$ is only negligible with respect to $\log N$; in the ratio it disappears exactly. The estimator $\widehat\theta_n$ uses $p=2$ because $c_2=\E Z_1^2=1$ for every $q$. With another power $p$, Proposition \ref{prop:power-u} would require the constant $c_p$, which depends on $q$ and has no closed form when $q\ge2$ and $p$ is not an even integer.

(ii) The observation is a single path at one spatial point, on a fixed time interval. This is a high frequency setting, in which $\theta$ is identified through the quadratic variation and not through a long time average.

(iii) If $H$ is unknown, one may replace $H$ by $\widehat H_n$ in $\widehat\theta_n$. Since $h_n^{1-2\widehat H_n}=h_n^{1-2H}\exp\big(2(H-\widehat H_n)\log(1/h_n)\big)$ and $\log(1/h_n)$ is of order $n$, the consistency of this plug-in estimator requires $n(\widehat H_n-H)\to0$ in probability, that is, a rate of convergence for $\widehat H_n$. For $q=1$ such rates follow from the limit theorems for quadratic variations of fractional Brownian motion, which are central for $H<3/4$ and non-central for $H>3/4$. For $q\ge2$ the fluctuations are non-Gaussian: for the Rosenblatt process see \cite{TudorViens2009}, and for every $q\ge2$, Chronopoulou, Tudor and Viens \cite[Theorem 1.1]{CTV2011} proved that the centred quadratic variation of $Z$, multiplied by $N^{(2-2H)/q}$ with $N$ the number of observations, converges in $L^2(\Omega)$ to a Rosenblatt random variable with Hurst index $1+2(H-1)/q$. This suggests the rate $N^{(2-2H)/q}$ for $\widehat H_n$. Its proof for the solution $u$ requires a quantitative version of Steps 1 and 2 of the proof of Proposition \ref{prop:power-u}, and is left for future work, as are the limit distributions of $\widehat H_n$ and $\widehat\theta_n$.
\end{remark}

% ======================================================================
\section{Relation with the Skorokhod integral}
\label{sec:skorokhod}
% ======================================================================

In this section we assume $q=2$, so that $Z$ is the Rosenblatt process (denoted $R^H$ in \cite{CDP2022}), and we relate the pathwise solution to the stochastic calculus of \cite{CDP2022}. We use their notation: $\int g_s\,\dd^-Z_s$ is the forward integral \cite[Definition 3]{CDP2022}, defined as the limit in probability of $\int g_s\,\eps^{-1}(Z_{s+\eps}-Z_s)\,\dd s$ as $\eps\downarrow0$ \cite{RussoVallois1993}; $\int g_s\,\delta Z_s$ is the Skorokhod integral with respect to $Z$ \cite[Definition 5]{CDP2022}; $B^{H/2+1/2}$ is the fractional Brownian motion with Hurst index $\frac H2+\frac12$ constructed from the same Wiener process $W$ \cite[Definition 1]{CDP2022}; $\int g_s\,\delta B_s^{H/2+1/2}$ is the corresponding Skorokhod integral \cite[Definition 4]{CDP2022}; $\nabla^{H/2}$ and $\nabla^{H/2,H/2}$ are the fractional stochastic derivatives of first and second order; and $c^R_H$, $c^{B,R}_H$ are the constants of \cite[Remark 1]{CDP2022}.

\begin{lemma}
\label{lem:forward-young}
Let $T>0$, $\omega\in\Omega_0$, and let $g\in C^{\gamma'}([0,T])$ with $\gamma'+\beta>1$. Then
\[
\lim_{\eps\downarrow0}\int_0^Tg_s\,\frac{Z_{s+\eps}-Z_s}{\eps}\,\dd s=\int_0^Tg_s\,\dd Z_s,
\]
where the right-hand side is a Young integral. (This holds for every $q\ge1$.) In particular, if $g$ is a process whose paths belong to $C^{\gamma'}([0,T])$ almost surely, then $g$ is forward integrable with respect to $Z$ on $[0,T]$, and its forward integral coincides almost surely with the pathwise Young integral.
\end{lemma}

\begin{proof}
Let $\varrho=Z(\omega)$ on $[0,T+1]$ and $\varrho^\eps_s=\eps^{-1}\int_s^{s+\eps}\varrho_r\,\dd r$ for $s\in[0,T]$ and $\eps\in(0,1]$. Then $\varrho^\eps\in C^1([0,T])$ with $\frac{\dd}{\dd s}\varrho^\eps_s=\eps^{-1}(\varrho_{s+\eps}-\varrho_s)$, so that
\[
\int_0^Tg_s\,\frac{\varrho_{s+\eps}-\varrho_s}{\eps}\,\dd s=\int_0^Tg\,\dd\varrho^\eps .
\]
Moreover $\|\varrho^\eps\|_{\beta;[0,T]}\le\|\varrho\|_{\beta;[0,T+1]}$ and $\sup_{s\le T}|\varrho^\eps_s-\varrho_s|\le\|\varrho\|_{\beta;[0,T+1]}\eps^\beta$. Choose $\beta'\in(1-\gamma',\beta)$. For a function $f$ and $s<t$, $|f(t)-f(s)|\le\min(\|f\|_\beta|t-s|^\beta,2\|f\|_\infty)\le(\|f\|_\beta|t-s|^\beta)^{\beta'/\beta}(2\|f\|_\infty)^{1-\beta'/\beta}$, hence
\[
\|\varrho^\eps-\varrho\|_{\beta';[0,T]}\le\big(2\|\varrho\|_{\beta;[0,T+1]}\big)^{\beta'/\beta}\Big(2\|\varrho\|_{\beta;[0,T+1]}\eps^\beta\Big)^{1-\beta'/\beta}\longrightarrow0 .
\]
By \eqref{eq:young-bound} with exponents $\gamma'+\beta'>1$,
\[
\left|\int_0^Tg\,\dd\varrho^\eps-\int_0^Tg\,\dd\varrho\right|\le\Big(|g_0|+C_{\gamma'+\beta'}\|g\|_{\gamma'}T^{\gamma'}\Big)\|\varrho^\eps-\varrho\|_{\beta';[0,T]}T^{\beta'}\longrightarrow0 .
\]
Almost sure convergence implies convergence in probability, which gives the second assertion.
\end{proof}

\begin{proposition}
\label{prop:skorokhod}
Assume (A1)--(A3) with $q=2$, let $u$ be the solution of Theorem \ref{thm:main}, and fix $(t,x)\in(0,\infty)\times\R$. Assume that the process $g_s=\sigma(u(s,x))$, $s\in[0,t]$, satisfies the assumptions 1 and 2 of \cite[Proposition 4]{CDP2022} on $[0,t]$. Then, almost surely,
\begin{align}
\label{eq:skorokhod-form}
u(t,x)&=u_0(x)+\int_0^t\Big[\nu u_{xx}(s,x)+u(s,x)u_x(s,x)\Big]\dd s+\int_0^t\sigma\big(u(s,x)\big)\,\delta Z_s\notag\\
&\quad+2c^{B,R}_H\int_0^t\big(\nabla^{H/2}g_s\big)(s)\,\delta B_s^{H/2+1/2}+c^R_H\int_0^t\big(\nabla^{H/2,H/2}g_s\big)(s,s)\,\dd s .
\end{align}
\end{proposition}

\begin{proof}
By Theorem \ref{thm:time-regularity}(a), the paths of $g$ belong to $C^{\beta}([0,t])$ almost surely, and $2\beta>1$. By Lemma \ref{lem:forward-young}, $g$ is forward integrable and $\int_0^tg_s\,\dd^-Z_s=\int_0^tg_s\,\dd Z_s$ almost surely. Under the stated assumptions, \cite[Proposition 4]{CDP2022} gives
\[
\int_0^tg_s\,\dd^-Z_s=\int_0^tg_s\,\delta Z_s+2c^{B,R}_H\int_0^t\big(\nabla^{H/2}g_s\big)(s)\,\delta B_s^{H/2+1/2}+c^R_H\int_0^t\big(\nabla^{H/2,H/2}g_s\big)(s,s)\,\dd s
\]
in $L^2(\Omega)$. Inserting this identity into \eqref{eq:integral-form} gives \eqref{eq:skorokhod-form}.
\end{proof}

\begin{remark}
\label{rem:skorokhod-open}
Proposition \ref{prop:skorokhod} shows that the pathwise solution does not solve the equation in which the stochastic term is a Skorokhod integral: two correction terms appear, a Skorokhod integral with respect to $B^{H/2+1/2}$ and a Lebesgue integral involving the second fractional stochastic derivative of $\sigma(u)$. The assumptions of \cite[Proposition 4]{CDP2022} require $g_s\in\D^{2,2}$, integrability conditions on $\nabla^{H/2}g_s$ and $\nabla^{H/2,H/2}g_s$, and continuity conditions near the diagonal. By Theorem \ref{thm:main}(a), $u(s,x)$ is a functional of the path $(Z_r)_{r\le s}$, and this functional depends on the path only through the coefficients of \eqref{eq:v}, evaluated at $Z_r$. For stochastic differential equations driven by processes in a finite Wiener chaos, Loosveldt, Nachit and Nourdin \cite{LNN2026} have recently obtained the Malliavin differentiability of pathwise solutions, by means of the approach of Kusuoka and Stroock and of a Taylor expansion of multiple integrals under Cameron--Martin shifts; see also \cite{NualartSaussereau2009} in the Gaussian case. Their method should apply to the map $\varrho\mapsto u(s,x)$ of Section \ref{sec:pde}. The integrability of the resulting derivatives is, however, a delicate point: the available pathwise bounds on derivatives of the solution are of the type described in Remark \ref{rem:gradient-moments}, and their expectations are infinite because of Remark \ref{rem:tails}. The verification of the assumptions of \cite[Proposition 4]{CDP2022} for the solution of \eqref{eq:burgers} is left open. For $q\ge3$, an analogue of the forward/Skorokhod relation is conjectured in \cite{CDP2022}, with stochastic integrals with respect to Hermite processes of lower orders, but it is not available yet. The divergence-type integral with respect to $Z$ itself is now defined for every $q$ \cite{CKS2025}, so the missing ingredient for $q\ge3$ is this correction formula, not the integral.
\end{remark}

% ======================================================================
\section{Concluding remarks}
\label{sec:conclusion}
% ======================================================================
We have shown that, for a multiplicative noise depending on time only, the stochastic Burgers equation driven by a Hermite process of any order $q\ge1$ with $H>1/2$ is well posed in a pathwise sense, and that this solution is the limit of the solutions driven by smooth approximations of the noise. The solution has finite moments of all orders and stretched exponential integrability of order $2/q$. For the Rosenblatt process, the exponential moments display an asymmetry that comes from the spectral structure of the second chaos. The temporal H\"older exponent equals $H$ for every $q$, and the power variations of $t\mapsto u(t,x)$ at a single point $x$ give consistent estimators of $H$ and of the intensity of the noise.

Several questions remain open. The first one is the verification of the assumptions of the forward/Skorokhod correction formula of \cite{CDP2022} for the solution when $q=2$ (Remark \ref{rem:skorokhod-open}), which is related to moment bounds for the gradient of the solution (Remark \ref{rem:gradient-moments}). The second one is the sharp description of the tails of $u(t,x)$ for $q\ge3$ (Remark \ref{rem:gaussian-comparison}). The third one concerns the estimators of Section \ref{sec:estimation}: their rates of convergence and limit distributions, which for $q\ge2$ should be non-Gaussian, and the plug-in estimator of $\theta$ when $H$ is unknown (Remark \ref{rem:estimation}). Finally, for a noise depending on space and time, such as a Hermite sheet, the transformation of Section \ref{sec:transform} is no longer available, and a genuinely infinite-dimensional theory is needed.

% ======================================================================
\appendix
\section{H\"older estimates for the heat equation on the line}
\label{app:heat}
% ======================================================================

In this appendix we prove the estimates for the heat equation used in Lemma \ref{lem:regularity}. We only need the operator $\partial_t-\nu\partial^2_{xx}$ with constant coefficients on $\R$, so everything can be done by explicit computations with the Gaussian kernel. Let $\tau>0$ and $Q_\tau=[0,\tau]\times\R$. For a function $g$ on $Q_\tau$ and $\alpha\in(0,1]$ we write
\[
[g]_{\alpha;x}=\sup_{t\in[0,\tau]}\ \sup_{x\neq y}\frac{|g(t,x)-g(t,y)|}{|x-y|^\alpha},\qquad
[g]_{\alpha/2;t}=\sup_{x\in\R}\ \sup_{s\neq t}\frac{|g(t,x)-g(s,x)|}{|t-s|^{\alpha/2}} ,
\]
so that $C^{\alpha/2,\alpha}_b(Q_\tau)$ is the space of bounded functions with $[g]_{\alpha;x}+[g]_{\alpha/2;t}<\infty$. For a function $f$ on $\R$, $[f]_\alpha$ is the usual H\"older seminorm.

\subsection{The kernel}

Let $\Gamma(w)=(4\pi)^{-1/2}e^{-w^2/4}$. Then $G_r(z)=(\nu r)^{-1/2}\,\Gamma\big(z(\nu r)^{-1/2}\big)$ for $r>0$, and, by the chain rule,
\begin{equation}
\label{eq:G-scaling}
\partial_z^kG_r(z)=(\nu r)^{-\frac{k+1}{2}}\,\Gamma^{(k)}\big(z(\nu r)^{-1/2}\big),\qquad k\ge0 .
\end{equation}
For $k\ge0$ and $\alpha\ge0$ let
\[
a_{k,\alpha}=\int_\R|\Gamma^{(k)}(w)|\,|w|^\alpha\,\dd w,\qquad
A_\alpha=\int_\R\ \sup_{|\lambda|\le1}\big|\Gamma'''(w+\lambda)\big|\,|w|^\alpha\,\dd w .
\]
These numbers are finite, because $\Gamma^{(k)}(w)=P_k(w)e^{-w^2/4}$ for a polynomial $P_k$ of degree $k$, and, for $|\lambda|\le1$, $(w+\lambda)^2\ge w^2/2-1$, so that $\sup_{|\lambda|\le1}|\Gamma'''(w+\lambda)|\le C(1+|w|)^3e^{1/4}e^{-w^2/8}$.

\begin{lemma}
\label{lem:kernel}
For $r>0$, $k\ge0$ and $\alpha\ge0$:
\begin{enumerate}[label=(\alph*)]
\item $\displaystyle\int_\R|\partial_z^kG_r(z)|\,|z|^\alpha\,\dd z=a_{k,\alpha}\,(\nu r)^{\frac{\alpha-k}{2}}$;
\item $\displaystyle\int_\R\partial_z^kG_r(z)\,\dd z=0$ for $k\ge1$, and $\displaystyle\int_\R G_r(z)\,\dd z=1$;
\item $\partial_rG_r=\nu\,\partial^2_{zz}G_r$, hence $\partial_r\partial^k_zG_r=\nu\,\partial^{k+2}_zG_r$;
\item for $\delta\le(\nu r)^{1/2}$ and $x\in\R$, $\displaystyle\int_\R\ \sup_{|\xi-x|\le\delta}\big|\partial^3_zG_r(\xi-y)\big|\,|x-y|^\alpha\,\dd y\le A_\alpha\,(\nu r)^{\frac{\alpha-3}{2}}$.
\end{enumerate}
\end{lemma}

\begin{proof}
(a) By \eqref{eq:G-scaling} and the change of variables $z=(\nu r)^{1/2}w$,
\[
\int_\R|\partial_z^kG_r(z)|\,|z|^\alpha\,\dd z
=(\nu r)^{-\frac{k+1}{2}}\int_\R|\Gamma^{(k)}(w)|\,(\nu r)^{\frac\alpha2}|w|^\alpha\,(\nu r)^{\frac12}\dd w
=a_{k,\alpha}(\nu r)^{\frac{\alpha-k}{2}} .
\]
(b) $\partial^{k-1}_zG_r$ tends to $0$ at $\pm\infty$, and $\int\Gamma=1$.
(c) This is a direct computation: with $G_r(z)=(4\pi\nu r)^{-1/2}e^{-z^2/(4\nu r)}$,
\[
\partial_rG_r=\Big(-\frac1{2r}+\frac{z^2}{4\nu r^2}\Big)G_r,\qquad
\partial^2_{zz}G_r=\Big(-\frac1{2\nu r}+\frac{z^2}{4\nu^2r^2}\Big)G_r=\frac1\nu\,\partial_rG_r .
\]
(d) Put $z=x-y$ and $\xi-y=z+(\xi-x)$, with $|\xi-x|\le\delta\le(\nu r)^{1/2}$. Writing $w=z(\nu r)^{-1/2}$ and $\lambda=(\xi-x)(\nu r)^{-1/2}\in[-1,1]$, \eqref{eq:G-scaling} gives $|\partial^3_zG_r(\xi-y)|=(\nu r)^{-2}|\Gamma'''(w+\lambda)|$. Hence the integral is at most
\[
(\nu r)^{-2}\int_\R\ \sup_{|\lambda|\le1}|\Gamma'''(w+\lambda)|\,(\nu r)^{\frac\alpha2}|w|^\alpha\,(\nu r)^{\frac12}\dd w=A_\alpha(\nu r)^{\frac{\alpha-3}{2}} .
\qedhere
\]
\end{proof}

Two elementary consequences will be used repeatedly. For $f\in\BUC(\R)$ with $[f]_\alpha<\infty$ for some $\alpha\in(0,1]$, and $r>0$,
\begin{equation}
\label{eq:S-hoelder}
[S(r)f]_\alpha\le[f]_\alpha,\qquad\|S(r)f-f\|_\infty\le a_{0,\alpha}[f]_\alpha\,(\nu r)^{\alpha/2}.
\end{equation}
Indeed, $|S(r)f(x)-S(r)f(x')|\le\int G_r(z)|f(x-z)-f(x'-z)|\,\dd z\le[f]_\alpha|x-x'|^\alpha$, and, by Lemma \ref{lem:kernel}(a),(b),
\[
|S(r)f(x)-f(x)|=\Big|\int_\R G_r(z)\big(f(x-z)-f(x)\big)\dd z\Big|\le[f]_\alpha\int_\R G_r(z)|z|^\alpha\,\dd z=a_{0,\alpha}[f]_\alpha(\nu r)^{\alpha/2}.
\]

\subsection{The initial datum}

\begin{lemma}
\label{lem:heat-initial}
Let $\gamma\in(0,1)$, $f\in C^{2+\gamma}_b(\R)$ and $w(t,x)=S(t)f(x)$. Then $w\in C^{1+\gamma/2,2+\gamma}_b(Q_\tau)$, with $w_x=S(t)f'$, $w_{xx}=S(t)f''$ and $w_t=\nu S(t)f''$. Moreover, $|w_x(t,x)-w_x(s,y)|\le\|f''\|_\infty\big(|x-y|+a_{0,1}(\nu|t-s|)^{1/2}\big)$.
\end{lemma}

\begin{proof}
Since $w(t,x)=\int G_t(z)f(x-z)\,\dd z$ and $f,f',f''$ are bounded, differentiation under the integral gives $w_x=S(t)f'$ and $w_{xx}=S(t)f''$. For $t>0$, Lemma \ref{lem:kernel}(c) and two integrations by parts in $z$ give
\begin{align*}
w_t(t,x)&=\int_\R\partial_tG_t(z)f(x-z)\,\dd z=\nu\int_\R\partial^2_{zz}G_t(z)f(x-z)\,\dd z\\
&=\nu\int_\R G_t(z)f''(x-z)\,\dd z=\nu S(t)f''(x).
\end{align*}
The boundary terms vanish because $f,f'$ are bounded and $G_t,\partial_zG_t$ decay at infinity. By \eqref{eq:S-hoelder} with $\alpha=\gamma$, $S(t)f''\to f''$ uniformly as $t\to0$, so $w(t,x)-f(x)=\nu\int_0^tS(s)f''(x)\,\dd s$, and $w_t$ exists and is continuous on $Q_\tau$, including at $t=0$. The function $h=S(\cdot)f''$ satisfies $[h]_{\gamma;x}\le[f'']_\gamma$ by \eqref{eq:S-hoelder}. For $0\le s<t$, the semigroup property and \eqref{eq:S-hoelder} give
\[
|h(t,x)-h(s,x)|=\big|S(s)\big(S(t-s)f''-f''\big)(x)\big|\le\|S(t-s)f''-f''\|_\infty\le a_{0,\gamma}[f'']_\gamma\big(\nu(t-s)\big)^{\gamma/2}.
\]
Hence $w_{xx}$ and $w_t$ belong to $C^{\gamma/2,\gamma}_b(Q_\tau)$. The last claim follows in the same way from \eqref{eq:S-hoelder} with $\alpha=1$ applied to $f'$, whose Lipschitz constant is $\|f''\|_\infty$:
\begin{align*}
|w_x(t,x)-w_x(s,y)|&\le|S(t)f'(x)-S(t)f'(y)|+\big\|S(s)\big(S(t-s)f'-f'\big)\big\|_\infty\\
&\le\|f''\|_\infty|x-y|+a_{0,1}\|f''\|_\infty\big(\nu(t-s)\big)^{1/2}.
\qedhere
\end{align*}
\end{proof}

\subsection{The heat potential with a bounded source}

For a bounded continuous function $g$ on $Q_\tau$ let
\[
Pg(t,x)=\int_0^t\big(S(t-s)g(s)\big)(x)\,\dd s=\int_0^t\int_\R G_{t-s}(x-y)\,g(s,y)\,\dd y\,\dd s .
\]
By Fubini's theorem and $S(t-s)=S(t-s')S(s'-s)$ for $s\le s'\le t$,
\begin{equation}
\label{eq:P-semigroup}
Pg(t)=S(t-s')\,Pg(s')+\int_{s'}^tS(t-r)g(r)\,\dd r,\qquad0\le s'\le t\le\tau .
\end{equation}

\begin{lemma}
\label{lem:heat-bounded}
Let $g\in C_b(Q_\tau)$ and $\theta'\in(0,1)$. Then $Pg$ and $\partial_xPg$ exist and are bounded and continuous on $Q_\tau$, with
\begin{equation}
\label{eq:Px}
\partial_xPg(t,x)=\int_0^t\int_\R\partial_xG_{t-s}(x-y)\,g(s,y)\,\dd y\,\dd s,\qquad\|\partial_xPg\|_\infty\le2\Big(\frac{\tau}{\pi\nu}\Big)^{1/2}\|g\|_\infty .
\end{equation}
Moreover there is a constant $C$, depending only on $\nu$, $\tau$ and $\theta'$, such that, for $(t,x),(s,y)\in Q_\tau$,
\begin{align}
|Pg(t,x)-Pg(s,y)|&\le C\|g\|_\infty\big(|x-y|+|t-s|^{1/2}\big),\label{eq:P-hoelder}\\
|\partial_xPg(t,x)-\partial_xPg(s,y)|&\le C\|g\|_\infty\big(|x-y|^{\theta'}+|t-s|^{\theta'/2}\big).\label{eq:Px-hoelder}
\end{align}
\end{lemma}

\begin{proof}
\emph{Step 1: the derivative.} Let $q(t,x)$ denote the right-hand side of the first identity in \eqref{eq:Px}. By Lemma \ref{lem:kernel}(a) with $k=1$, $\alpha=0$, and $a_{1,0}=2\Gamma(0)=\pi^{-1/2}$,
\[
\int_0^t\int_\R|\partial_xG_{t-s}(x-y)|\,|g(s,y)|\,\dd y\,\dd s\le\|g\|_\infty\int_0^t\big(\pi\nu(t-s)\big)^{-1/2}\dd s=2\Big(\frac t{\pi\nu}\Big)^{1/2}\|g\|_\infty .
\]
Thus $q$ is well defined and satisfies the bound in \eqref{eq:Px}. For $x_0<x$, Fubini's theorem, justified by the same bound, and $\int_{x_0}^x\partial_\xi G_r(\xi-y)\,\dd\xi=G_r(x-y)-G_r(x_0-y)$ give
\[
\int_{x_0}^xq(t,\xi)\,\dd\xi=Pg(t,x)-Pg(t,x_0).
\]
We show in Step 2 that $q(t,\cdot)$ is continuous. Then $\partial_xPg=q$ by the fundamental theorem of calculus.

\emph{Step 2: H\"older continuity of $q$ in $x$.} Let $t\in(0,\tau]$, $x\neq x'$ and $\delta=|x-x'|$. By the mean value theorem and Lemma \ref{lem:kernel}(a),
\[
\int_\R|\partial_zG_r(z+\delta)-\partial_zG_r(z)|\,\dd z\le\min\Big(2a_{1,0}(\nu r)^{-1/2},\ \delta\,a_{2,0}(\nu r)^{-1}\Big).
\]
Hence, with $r=t-s$ and $r_\delta=\delta^2/\nu$,
\begin{align*}
|q(t,x)-q(t,x')|
&\le\|g\|_\infty\int_0^t\min\Big(2a_{1,0}(\nu r)^{-1/2},\ \delta\,a_{2,0}(\nu r)^{-1}\Big)\dd r\\
&\le\|g\|_\infty\Big(2a_{1,0}\int_0^{r_\delta}(\nu r)^{-1/2}\dd r+\delta\,a_{2,0}\int_{r_\delta}^{\max(t,r_\delta)}(\nu r)^{-1}\dd r\Big)\\
&=\|g\|_\infty\Big(\frac{4a_{1,0}}{\nu}\,\delta+\frac{a_{2,0}}{\nu}\,\delta\log_+\frac{\nu t}{\delta^2}\Big).
\end{align*}
If $\delta\le(\nu\tau)^{1/2}$, the function $\delta\mapsto\delta^{1-\theta'}\big(1+\log(\nu\tau/\delta^2)\big)$ is bounded on $(0,(\nu\tau)^{1/2}]$, so the right-hand side is at most $C\|g\|_\infty\delta^{\theta'}$. If $\delta>(\nu\tau)^{1/2}$, the bound in \eqref{eq:Px} gives $|q(t,x)-q(t,x')|\le4(\tau/(\pi\nu))^{1/2}\|g\|_\infty\le C\|g\|_\infty\delta^{\theta'}$. Thus $[q]_{\theta';x}\le C\|g\|_\infty$.

\emph{Step 3: H\"older continuity in $t$.} Let $0\le s<t\le\tau$ and $h=t-s$. By \eqref{eq:P-semigroup} with $s'=s$, and by $\partial_xS(h)\varphi=S(h)\varphi'$ for $\varphi\in\BUC^1(\R)$,
\[
q(t)=S(h)q(s)+\int_s^t\partial_xS(t-r)g(r)\,\dd r .
\]
By \eqref{eq:S-hoelder} with $\alpha=\theta'$ and Step 2, $\|S(h)q(s)-q(s)\|_\infty\le a_{0,\theta'}C\|g\|_\infty(\nu h)^{\theta'/2}$. By \eqref{eq:heat-estimates}, the integral is bounded by $\|g\|_\infty\int_0^h(\pi\nu r)^{-1/2}\dd r=2(h/(\pi\nu))^{1/2}\|g\|_\infty$. Since $h\le\tau$, $h^{1/2}\le\tau^{(1-\theta')/2}h^{\theta'/2}$. Together with Step 2 this proves \eqref{eq:Px-hoelder}, and in particular $q=\partial_xPg$ is continuous on $Q_\tau$. In the same way, \eqref{eq:P-semigroup}, \eqref{eq:S-hoelder} with $\alpha=1$ and the Lipschitz bound $\|\partial_xPg(s)\|_\infty\le2(\tau/(\pi\nu))^{1/2}\|g\|_\infty$ give
\[
|Pg(t,x)-Pg(s,x)|\le a_{0,1}\|\partial_xPg(s)\|_\infty(\nu h)^{1/2}+h\|g\|_\infty ,
\]
which, combined with the Lipschitz bound in $x$, proves \eqref{eq:P-hoelder}.
\end{proof}

\subsection{Schauder estimates for the heat potential}

\begin{lemma}
\label{lem:heat-schauder}
Let $\alpha\in(0,1)$ and $g\in C^{\alpha/2,\alpha}_b(Q_\tau)$ be continuous. Then $Pg\in C^{1+\alpha/2,2+\alpha}_b(Q_\tau)$, $Pg(0,\cdot)=0$, and $\partial_tPg=\nu\,\partial^2_{xx}Pg+g$ on $Q_\tau$. Moreover
\begin{equation}
\label{eq:Pxx}
\partial^2_{xx}Pg(t,x)=\int_0^t\int_\R\partial^2_{x}G_{t-s}(x-y)\big[g(s,y)-g(s,x)\big]\dd y\,\dd s ,
\end{equation}
and there is a constant $C$, depending only on $\nu$, $\tau$ and $\alpha$, such that $\|\partial^2_{xx}Pg\|_\infty+[\partial^2_{xx}Pg]_{\alpha;x}+[\partial^2_{xx}Pg]_{\alpha/2;t}\le C[g]_{\alpha;x}$.
\end{lemma}

\begin{proof}
Write $[g]=[g]_{\alpha;x}$ and let $w_2(t,x)$ be the right-hand side of \eqref{eq:Pxx}. By Lemma \ref{lem:kernel}(a) with $k=2$,
\begin{equation}
\label{eq:w2-integrand}
\int_\R\big|\partial^2_xG_r(x-y)\big|\,|g(s,y)-g(s,x)|\,\dd y\le[g]\int_\R|\partial^2_zG_r(z)|\,|z|^\alpha\,\dd z=a_{2,\alpha}[g](\nu r)^{\frac\alpha2-1},
\end{equation}
and $\int_0^t(\nu r)^{\frac\alpha2-1}\dd r=\frac2{\alpha\nu}(\nu t)^{\alpha/2}$. Hence $w_2$ is well defined and $\|w_2\|_\infty\le\frac{2a_{2,\alpha}}{\alpha\nu}(\nu\tau)^{\alpha/2}[g]$.

\emph{Step 1: $\partial^2_{xx}Pg=w_2$.} For $\eps\in(0,\tau)$ let $P_\eps g(t,x)=\int_0^{(t-\eps)_+}\int_\R G_{t-s}(x-y)g(s,y)\,\dd y\,\dd s$. Since $t-s\ge\eps$ in this integral, $G_{t-s}$ and all its derivatives are bounded and integrable uniformly in $s$, so we may differentiate under the integral. Using Lemma \ref{lem:kernel}(b) to subtract $g(s,x)$, we obtain
\begin{align*}
\partial_xP_\eps g(t,x)&=\int_0^{(t-\eps)_+}\!\!\int_\R\partial_xG_{t-s}(x-y)g(s,y)\,\dd y\,\dd s,\\
\partial^2_{xx}P_\eps g(t,x)&=\int_0^{(t-\eps)_+}\!\!\int_\R\partial^2_xG_{t-s}(x-y)\big[g(s,y)-g(s,x)\big]\dd y\,\dd s .
\end{align*}
By \eqref{eq:Px} and the bound $\int|\partial_xG_r|=(\pi\nu r)^{-1/2}$, $\|\partial_xP_\eps g-\partial_xPg\|_\infty\le\|g\|_\infty\int_0^\eps(\pi\nu r)^{-1/2}\dd r\to0$. By \eqref{eq:w2-integrand}, $\|\partial^2_{xx}P_\eps g-w_2\|_\infty\le a_{2,\alpha}[g]\int_0^\eps(\nu r)^{\frac\alpha2-1}\dd r\to0$. Since $\partial_xP_\eps g\to\partial_xPg$ and $\partial_x(\partial_xP_\eps g)\to w_2$ uniformly on $Q_\tau$, $\partial_xPg$ is differentiable in $x$ and $\partial^2_{xx}Pg=w_2$.

\emph{Step 2: H\"older continuity of $w_2$ in $x$.} Fix $t$, $x\neq x'$, $\delta=|x-x'|$, and write $g_s=g(s,\cdot)$, $r=t-s$, and
\[
I(r;x)=\int_\R\partial^2_xG_r(x-y)\big[g_s(y)-g_s(x)\big]\dd y ,
\]
so that $w_2(t,x)=\int_0^tI(t-s;x)\,\dd s$. We split the $s$-integral according to $\nu r\le\delta^2$ or $\nu r>\delta^2$.

If $\nu r\le\delta^2$, we bound $I(r;x)$ and $I(r;x')$ separately by \eqref{eq:w2-integrand}:
\[
\int_{\{s:\ \nu(t-s)\le\delta^2\}}\big(|I(t-s;x)|+|I(t-s;x')|\big)\dd s\le2a_{2,\alpha}[g]\int_0^{\delta^2/\nu}(\nu r)^{\frac\alpha2-1}\dd r=\frac{4a_{2,\alpha}}{\alpha\nu}[g]\,\delta^\alpha .
\]
If $\nu r>\delta^2$, we use Lemma \ref{lem:kernel}(b) to replace $g_s(x')$ by $g_s(x)$ in $I(r;x')$:
\[
I(r;x)-I(r;x')=\int_\R\Big[\partial^2_xG_r(x-y)-\partial^2_xG_r(x'-y)\Big]\big[g_s(y)-g_s(x)\big]\dd y .
\]
By the mean value theorem, $|\partial^2_xG_r(x-y)-\partial^2_xG_r(x'-y)|\le\delta\sup_{|\xi-x|\le\delta}|\partial^3_zG_r(\xi-y)|$. Since $\delta<(\nu r)^{1/2}$, Lemma \ref{lem:kernel}(d) gives
\[
|I(r;x)-I(r;x')|\le\delta\,[g]\int_\R\sup_{|\xi-x|\le\delta}|\partial^3_zG_r(\xi-y)|\,|x-y|^\alpha\,\dd y\le A_\alpha[g]\,\delta\,(\nu r)^{\frac{\alpha-3}{2}} .
\]
Since $\alpha<1$,
\[
\int_{\delta^2/\nu}^\infty(\nu r)^{\frac{\alpha-3}{2}}\dd r=\frac1\nu\cdot\frac2{1-\alpha}\,\delta^{\alpha-1},
\]
so that
\[
\int_{\{s:\ \nu(t-s)>\delta^2\}}|I(t-s;x)-I(t-s;x')|\,\dd s\le\frac{2A_\alpha}{(1-\alpha)\nu}[g]\,\delta^\alpha .
\]
Adding the two parts, $[w_2]_{\alpha;x}\le\frac1\nu\big(\frac{4a_{2,\alpha}}\alpha+\frac{2A_\alpha}{1-\alpha}\big)[g]$.

\emph{Step 3: H\"older continuity of $w_2$ in $t$.} Fix $x$, $0\le s<t\le\tau$, $h=t-s$, and write $D(r,y)=g(r,y)-g(r,x)$. Then
\begin{align*}
w_2(t,x)-w_2(s,x)
&=\int_s^t\int_\R\partial^2_xG_{t-r}(x-y)D(r,y)\,\dd y\,\dd r\\
&\quad+\int_0^s\int_\R\Big[\partial^2_xG_{t-r}(x-y)-\partial^2_xG_{s-r}(x-y)\Big]D(r,y)\,\dd y\,\dd r
=:J_1+J_2 .
\end{align*}
By \eqref{eq:w2-integrand}, $|J_1|\le a_{2,\alpha}[g]\int_0^h(\nu\rho)^{\frac\alpha2-1}\dd\rho=\frac{2a_{2,\alpha}}{\alpha\nu}[g](\nu h)^{\alpha/2}$. In $J_2$ we put $\rho=s-r\in(0,s]$, so that $t-r=\rho+h$, and split according to $\rho\le h$ or $\rho>h$. For $\rho\le h$ we bound the two kernels separately by \eqref{eq:w2-integrand}, using $\rho+h\ge h$:
\begin{align*}
\int_0^{\min(h,s)}a_{2,\alpha}[g]\Big((\nu(\rho+h))^{\frac\alpha2-1}+(\nu\rho)^{\frac\alpha2-1}\Big)\dd\rho
&\le a_{2,\alpha}[g]\Big(h(\nu h)^{\frac\alpha2-1}+\frac2{\alpha\nu}(\nu h)^{\frac\alpha2}\Big)\\
&=\frac{a_{2,\alpha}}{\nu}\Big(1+\frac2\alpha\Big)[g](\nu h)^{\alpha/2}.
\end{align*}
For $\rho>h$, Lemma \ref{lem:kernel}(c) gives $\partial^2_zG_{\rho+h}-\partial^2_zG_\rho=\nu\int_\rho^{\rho+h}\partial^4_zG_u\,\dd u$, and Lemma \ref{lem:kernel}(a) with $k=4$ gives
\[
\int_\R\big|\partial^2_zG_{\rho+h}(z)-\partial^2_zG_\rho(z)\big|\,|z|^\alpha\,\dd z\le\nu\int_\rho^{\rho+h}a_{4,\alpha}(\nu u)^{\frac{\alpha-4}{2}}\dd u\le\nu h\,a_{4,\alpha}(\nu\rho)^{\frac{\alpha-4}{2}} .
\]
Since $\int_h^\infty(\nu\rho)^{\frac{\alpha-4}{2}}\dd\rho=\frac1\nu\cdot\frac2{2-\alpha}(\nu h)^{\frac{\alpha-2}{2}}$, the part $\rho>h$ of $J_2$ is bounded by
\[
[g]\,\nu h\,a_{4,\alpha}\cdot\frac1\nu\cdot\frac2{2-\alpha}(\nu h)^{\frac{\alpha-2}{2}}=\frac{2a_{4,\alpha}}{(2-\alpha)\nu}[g](\nu h)^{\alpha/2}.
\]
Hence $[w_2]_{\alpha/2;t}\le C[g]$.

\emph{Step 4: the time derivative.} For $t>\eps$, differentiation of $P_\eps g$ in $t$ gives, by Lemma \ref{lem:kernel}(c),
\[
\partial_tP_\eps g(t,x)=\big(S(\eps)g(t-\eps)\big)(x)+\nu\,\partial^2_{xx}P_\eps g(t,x).
\]
By \eqref{eq:S-hoelder},
\[
\big|S(\eps)g(t-\eps)(x)-g(t,x)\big|\le\|g(t-\eps)-g(t)\|_\infty+\|S(\eps)g(t)-g(t)\|_\infty\le[g]_{\alpha/2;t}\,\eps^{\alpha/2}+a_{0,\alpha}[g](\nu\eps)^{\alpha/2}.
\]
Together with Step 1, $\partial_tP_\eps g\to g+\nu w_2$ uniformly on $[\eta,\tau]\times\R$ for every $\eta>0$, while $P_\eps g\to Pg$ uniformly. Hence $Pg$ is $C^1$ in $t$ on $(0,\tau]$, with $\partial_tPg=g+\nu w_2$. Since $g+\nu w_2$ is continuous and bounded on $Q_\tau$ (Steps 2 and 3) and $Pg(\eta,x)\to0$ as $\eta\to0$, we get $Pg(t,x)=\int_0^t(g+\nu w_2)(s,x)\,\dd s$, so the identity also holds at $t=0$. Finally $\partial_tPg=g+\nu\partial^2_{xx}Pg\in C^{\alpha/2,\alpha}_b(Q_\tau)$, because both terms are, and $Pg$, $\partial_xPg$ are bounded and continuous by Lemma \ref{lem:heat-bounded}.
\end{proof}

\end{document}